\documentclass[12pt]{amsart}

\usepackage[hmargin=0.8in,height=8.6in]{geometry}
\usepackage{amssymb,amsthm, times}
\usepackage{delarray,verbatim}
\usepackage{graphicx}
\usepackage{ifpdf}
\usepackage{color}
\definecolor{webgreen}{rgb}{0,.5,0}
\definecolor{webbrown}{rgb}{.8,0,0}
\definecolor{emphcolor}{rgb}{0.95,0.95,0.95}

\usepackage{hyperref}
\hypersetup{%
          colorlinks=true,
          linkcolor=webbrown,
          filecolor=webbrown,
          citecolor=webgreen,
          breaklinks=true}
\ifpdf \hypersetup{pdftex,
            pdfstartview=FitH, %%Fit, FitB, FitH
            bookmarksopen=true,
            bookmarksnumbered=true
} \else \hypersetup{dvips} \fi

\newcommand {\e}{\mathbb{E}}

\numberwithin{equation}{section}

\newtheorem{theorem}{Theorem}[section]

\newtheorem{remark}{Remark}[section]
\newtheorem{lemma}{Lemma}[section]

\numberwithin{remark}{section} \numberwithin{proposition}{section}
\numberwithin{corollary}{section}
\newcommand {\R}{\mathbb{R}}

\newcommand {\p}{\mathbb{P}}

\newcommand{\diff}{{\rm d}}

\title{Optimal double stopping of a Brownian bridge}
\thanks{This version: \today.}

\author[E. J. Baurdoux]{Erik J. Baurdoux}
\address[E. J. Baurdoux]{ Department of Statistics, London School of Economics, Houghton Street, London, WC2A 2AE, UK. }
\email{e.j.baurdoux@lse.ac.uk}
\author[N. Chen]{Nan Chen}
\address[N. Chen]{Department of Systems Engineering and Engineering Management, the Chinese University of Hong Kong, Shatin, N.T., Hong Kong}
\email{nchen@se.cuhk.edu.hk}
\author[B. A. Surya]{\,\,Budhi A. Surya}

\address[B. A. Surya]{School of Business and Management, Bandung Institute of Technology, 10 Ganesha Street, 40132 Bandung, West Java, Indonesia.}
\email{budhi.surya@sbm-itb.ac.id}

\author[K. Yamazaki]{Kazutoshi Yamazaki}
\address[K. Yamazaki]{Department of Mathematics,
Faculty of Engineering Science, Kansai University, 3-3-35 Yamate-cho, Suita-shi, Osaka 564-8680, Japan. }
\email{kyamazak@kansai-u.ac.jp}

\date{}
\begin{document}

\maketitle \noindent

\begin{abstract}We study optimal double stopping problems driven by a Brownian bridge.   The objective is to maximize the expected spread between the payoffs achieved at the two stopping times.  We study several cases where the solutions can be solved explicitly by strategies of threshold type.
\\
\noindent \small{\textbf{Key words:} Brownian bridge; optimal double stopping, buying-selling strategies \\
%  \noindent \textbf{JEL Classification:} G32, D81, C61 \\
\noindent \textbf{Mathematics Subject Classification (2010):}  60G40, 60H30}\\
\end{abstract}

\section{Introduction}\label{sec:problem}

In this paper, we study several optimal double stopping problems for a Brownian bridge.  Given a Brownian bridge $\{X_s\}_{t\leq s\leq 1}$ starting from $x$ at time $0 \leq t < 1$ and ending at $0$ at time $1$, or equivalently a Brownian motion conditioned to be at $0$ at time $1$, our objective is to choose a pair of stopping times, $t \leq \tau_1 \leq \tau_2 <1$ such that the expected spread between the payoffs $f(X_{\tau_2})$ and $f(X_{\tau_1})$ is maximized for a given functional $f$.

The optimal double stopping problem has received much attention recently in the field of finance.  In particular, this is used to derive a ``buy low and sell high" strategy so as to maximize the expected spread between the two payoffs.  The strategy called \emph{mean-reversion} typically uses the ``mean" computed from the historical data as a benchmark;  an asset is bought if the price is lower  and is sold when it is higher. Closely related is the trading strategy called \emph{pairs trading}. Two assets of similar characteristics (e.g., in the same industry category) are considered.  By longing one and shorting the other, one can construct a mean-reverting portfolio. An implementation of a pairs trading reduces to solving a single or double stopping problem where one wants to decide the time of (entry and) liquidation of the position so as to maximize the spread.
We refer the reader to, e.g., \cite{ekstrom2011optimal}, \cite{Leung2013} and \cite{Song2009} among others.

There are several motivations to consider a Brownian bridge as an underlying process.  We list here three examples where an asset process is expected to converge to a given value at a given time, and hence a Brownian bridge is suitable in modeling.

The first example, known as the \emph{stock pinning}, is a phenomenon where  a stock price tends to  end up in the vicinity of the strike of its option near its expiry.  This is observed typically for heavily traded assets; within minutes before the expiration, the stock price experiences a strong mean-reversion to the strike.   We refer the reader to  \cite{Avellaneda} and \cite{Ekstrom} and references therein,  for the discussion on the mechanism of the stock pinning.

The second example is a sudden mispricing of assets due to the market's overreaction to news and rumors, which is followed by a rapid recovery to the original value. In the well-known 2010 \emph{Flash Crash},  the Dow Jones Industry Average fell about 9 percent and then recovered within minutes; see, e.g., Chapter III of \cite{Lin}.  While its cause is still in dispute, it is believed to have been first triggered by some newly disclosed information on debt crisis in Greece, followed by a chain reaction of large execution of sales by the automated algorithmic/high frequency trading. While the price may not recover completely to the original price, the difference is  small in comparison to the magnitude of the large fall caused by these events.

%
%The second example is the so-called  \emph{Flash Crash}, which is a recent phenomenon of a sudden and significant fall of a security price followed by its recovery in an extremely short period.     This significant mispricing is created within a short time and then reverts back to the original price.
%This price process can be well modeled by a Brownian bridge;

The third example comes from the dynamic prices of goods in the existence of seasonality and/or fixed sales deadlines.   Important examples include low cost carriers (LCC's)/high speed rails, hotel rooms and theater tickets, where these goods become worthless after given deadlines.  In the field of revenue/yield management, the price of such good is chosen dynamically (and stochastically) over time so as to maximize the expected total yield; the problem reduces to striking the balance between maximizing the price per unit and minimizing the remaining stocks at the deadline. In typical models, the dynamic programming principle applies and the optimal price becomes a function of the remaining number of stocks and the remaining time until the deadline; see, among others, the seminal paper by Gallego and van Ryzin \cite{Gallego}.  According to these models, the price converges to a given value on condition that the remaining inventory vanishes by the deadline; this is aimed by the manager and is indeed more than likely achieved when the demand is high (e.g., holiday seasons).
%
%
%As in the seminal paper by Gallego and Ryzin TODO, the managers maximize the total revenues by optimally choosing the price dynamically on which demand of the good depends on.
%As the vendor models the price so as to make the remaining inventory to vanish, the price process experiences a fluctuation similar to Brownian bridge where it converges to a given price at the time of expiry.

The optimal double stopping problem for a Brownian bridge considered in this paper is applicable in situations where one wants to buy and sell an asset to maximize the spread until it converges to the fixed value as in these examples.

There are papers on the single optimal stopping problem for a Brownian bridge.  In particular, Shepp \cite{Shepp} solves the problem of maximizing the first moment of the stopped Brownian bridge (under the assumption that it starts at zero) by rewriting the problem in terms of a time-changed Brownian motion.  Ekstr\"{o}m and Wanntorp \cite{Ekstrom} solve for several payoff functionals with arbitrary starting values.  Our findings heavily rely on the latter; we shall start with the results in \cite{Ekstrom} and extend to the optimal double stopping problem.  Regarding the discrete-time analog (the urn problem), we refer the reader to \cite{Mazalov_Tamaki}  and \cite{Tamaki} for single optimal stopping problems.  For optimal double stopping problems, Ivashko \cite{Ivashko} considers the problem of maximizing the spread of the first moment; Sofrenov et al.\
\cite{Sofronov} consider a different but related buying-selling problem under independent observations.

\subsection{Problems}
Fix $0 \leq t <1$ and consider a Brownian bridge $\{X_s\}_{t\leq s\leq 1}$ satisfying
\begin{align}
\mathrm{d}X_s=-\frac{X_s}{1-s}\mathrm{d}s+\mathrm{d}W_s,\quad t\leq s < 1, \label{BB_SDE}
\end{align}
with $X_t=x \in \R$ and where $\{ W_s\}_{t \leq s \leq 1}$ denotes a standard Brownian motion.   We let $\p_{t,x}$ and $\e_{t,x}$ be the conditional probability and expectation under which $X_t = x$ for any $0 \leq t < 1$ and $x \in \R$.

We consider three problems of maximizing the expected spread given as follows:
\begin{description}
\item[Problem 1]  $\e_{t,x} \left[X_{\tau_2} -  X_{\tau_1} \right]$,
\item[Problem 2]  $\e_{t,x}[(X_{\tau_2}^{2n+1} -  X_{\tau_1}^{2n+1}) 1_{\{X_{\tau_1} \leq 0\}} + (X_{\tau_1}^{2n+1} -  X_{\tau_2}^{2n+1}) 1_{\{X_{\tau_1} > 0\}}]$, for a given integer $n \geq 0$,
\item[Problem 3] $\e_{t,x} \left[|X_{\tau_2}|^{q} -  |X_{\tau_1}|^{q} \right]$,
for a given $q > 0$.
\end{description}
The supremum is taken over all pairs of stopping times $t \leq \tau_1 \leq \tau_2 < 1$ a.s.\ with respect to the filtration generated by $X$.

Problem 1 corresponds to the case where short-selling is not permitted, and an asset must be bought prior to being sold. Problem 2 is the case where it is allowed; if the price at the first exercise time is negative (resp.\ positive), the asset is bought (resp.\ sold) and then it is sold (resp.\ bought) at the second exercise time. Problem 3 models the case when the payoff function is v-shaped with respect to the underlying; this is motivated by investing strategies such as a straddle.

For each problem, we shall show that the optimal stopping times are first hitting times of the time-changed process $\{X_s/\sqrt{1-s}\}_{t \leq s \leq 1}$. 
%We also give a discussion on extensions of Problem 1 for higher moments; contrary to Problems 2 and 3, strategies of threshold type fail to be optimal.

To the best of our knowledge, this is the first result on the finite-time horizon optimal double stopping problem where the solution is nontrivial and explicit. It is remarked that a finite-time horizon optimal stopping in general lacks an explicit solution even for a single stopping case.
For other processes, we expect that the solutions  are either trivial (e.g.\ buying immediately and selling at the maturity) or do not admit analytical solutions.  It is also noted that thanks to the a.s.\ fixed end point of a Brownian bridge, the two stoppings are always exercised; for other processes, one needs to take care of a scenario where the first and/or second stoppings never occur during the time horizon.

\subsection{Outlines}

The rest of the paper is organized as follows. Section \ref{section_preliminaries} reviews the single optimal stopping problem of a Brownian bridge as obtained in \cite{Ekstrom} with some complements that will be needed for our analysis in later sections.  Sections  \ref{section_problem1},  \ref{section_problem2} and  \ref{section_problem3} solve Problems 1, 2 and 3, respectively.   Some proofs are deferred to Appendix \ref{appendix_proof}.

\section{Preliminaries} \label{section_preliminaries}
In this section, we review the results of Ekstr\"{o}m and Wanntorp \cite{Ekstrom} for the optimal single stopping problem of a Brownian bridge.  As there are a few details omitted in \cite{Ekstrom} but will be important in our analysis, we complement these results here.  Throughout, let us define, for all $q > 0$,\footnote{It is remarked that (3.5) of \cite{Ekstrom} contains a typo in their definitions of $F_{q}$ and $G_{q}$. We suggest the reader to refer to  Section 4 of Ekstr\"{o}m et al.  \cite{ekstrom2011optimal} for a correct version.}
\begin{align}
F_{q}(y) := \int_0^\infty u^{q-1} e^{yu -  {u^2} /2} \diff u \quad \textrm{and} \quad G_{q}(y) := F_{q}(-y), \quad y \in \R. \label{def_F_G}
\end{align}
These functions can be written in terms of the confluent hypergeometric/parabolic cylinder functions; see, e.g., \cite{Abramowitz}.
Consider the partial differential equation (PDE), $\mathcal{L} \xi(t,x) = 0$, for $\xi \in C^1 \times C^2$ on some open set $E$, with the infinitesimal generator $\mathcal{L}$ for a Brownian bridge \eqref{BB_SDE},
\begin{align}
\mathcal{L} \xi(t,x) := \frac \partial {\partial t}\xi -  \frac x {1-t} \frac \partial {\partial x} \xi + \frac 1 2 \frac {\partial^2} {\partial x^2} \xi, \quad (t,x) \in E. \label{pde}
\end{align}
This can be simplified by setting $\xi(t,x) = (1-t)^{q/2} \zeta(x/\sqrt{1-t})$ to an ordinary differential equation (ODE), 
\begin{align}
\zeta''(y) - y \zeta'(y) - q \zeta(y) = 0. \label{eq_ODE}
\end{align} 
A general solution of \eqref{eq_ODE} can be written as a linear combination of $F_{q}$ and $G_{q}$; see Section 4 of \cite{ekstrom2011optimal}.

In particular, when $q=1$, \eqref{def_F_G} is simplified to
\begin{align}
F_{1}(y) =   e^{{y^2} /2} \int_{-y}^\infty e^{- {u^2} / 2} \diff u =  \sqrt{2 \pi}  e^{{y^2} /2} \Phi (y) \quad \textrm{and} \quad G_{1}(y)  =  \sqrt{2 \pi}  e^{{y^2} /2} \Phi (-y), \quad y \in \R, \label{F_1_def}
\end{align}
where $\Phi$ denotes the standard normal distribution function, i.e.,
\[\Phi(y):=\frac{1}{\sqrt{2\pi}}\int_{-\infty}^y e^{-z^2/2}\,\mathrm{d}z, \quad y \in \R. \]
Consequently, we also have $(G_1 + F_{1})(y)  =  \sqrt{2 \pi}  \exp({y^2} /2)$ for all $y \in \R$.

\subsection{One-sided exit problem}
For fixed integer $n \geq 0$, consider the single stopping problem:
\begin{align}
U(t, x) := \sup_{t \leq \tau < 1}\e_{t,x} [ X_{\tau}^{2n+1}], \quad 0 \leq t < 1, \; x \in \R. \label{ospbridge_general}
\end{align}
Define the upcrossing time of the process $\{X_s/\sqrt{1-s}\}_{t \leq s \leq 1}$,
\begin{align}
\tau^+(B):=\inf\{s \geq t: X_s \geq B \sqrt{1-s} \}, \quad B \in \R. \label{tau_B}
\end{align}
Following the arguments as in \cite{Ekstrom}, we have, for any $B \in \R$,
\begin{align}
\e_{t,x} [X_{\tau^+(B)}^{2n+1}] = (1-t)^{n+1 / 2} \frac {B^{2n+1}} {F_{2n+1}(B)}{F_{2n+1} \Big(x / {\sqrt{1-t}}\Big)}, \quad x < B \sqrt{1-t}, \label{expectation_moments_2n_1}
\end{align}
which can be derived  by solving \eqref{eq_ODE} for $q = 2n+1$ and $y = x/ \sqrt{1-t}$ with its boundary conditions; see page 172 of \cite{Ekstrom}.

Ekstr\"{o}m and Wanntorp \cite{Ekstrom} show that \eqref{ospbridge_general} is solved by the stopping time \eqref{tau_B} by choosing
$B$ that maximizes \eqref{expectation_moments_2n_1} or equivalently the function $B \mapsto {B^{2n+1}}/ {F_{2n+1}(B)}$. Taking its derivative,
\begin{align*}
\frac \partial {\partial B}\frac {B^{2n+1}} {F_{2n+1}(B)} = \frac {B^{2n}} {F_{2n+1}(B)} \Big[ (2n+1)- \frac {B F_{2n+1}'(B)} {F_{2n+1}(B)} \Big], \quad B \in \R.
\end{align*}
The sign of the above is determined by that of the function 
\begin{align*}
B \mapsto (2n+1)- {B F_{2n+1}'(B)} / {F_{2n+1}(B)},
\end{align*}
 which is plotted in Figure \ref{plot_B_star}.
As is shown in \cite{Ekstrom}, it is monotonically decreasing and there exists a unique zero $B^* > 0$ such that
\begin{align}
B^* F'_{2n+1}(B^*) = (2n+1) F_{2n+1}(B^*) \label{B_star}
\end{align}
and
\begin{align}
\frac \partial {\partial B}\frac {B^{2n+1}} {F_{2n+1}(B)}> 0 \Longleftrightarrow B < B^*, \quad B \in \R. \label{monotonicity_B}
\end{align}

Define the candidate value function $U^*(t, x) := \e_{t,x} [ X_{\tau^+(B^*)}^{2n+1}]$ for $0 \leq t < 1$ and $x \in \R$. 
The verification of optimality requires the following lower bound on $B^*$; as it is not included in \cite{Ekstrom}, we shall give its proof.  Note that this is also confirmed in the numerical plots of Figure \ref{plot_B_star}.

%The following result is important but is not included in \cite{Ekstrom}; it should be needed for verification.
\begin{lemma} \label{lemma_root_n}We have $B^* \geq \sqrt{n}$.
\end{lemma}
\begin{proof} See Appendix \ref{appendix_proof}.
\end{proof}

\begin{figure}[htbp]
\begin{center}
 \includegraphics[scale=0.65]{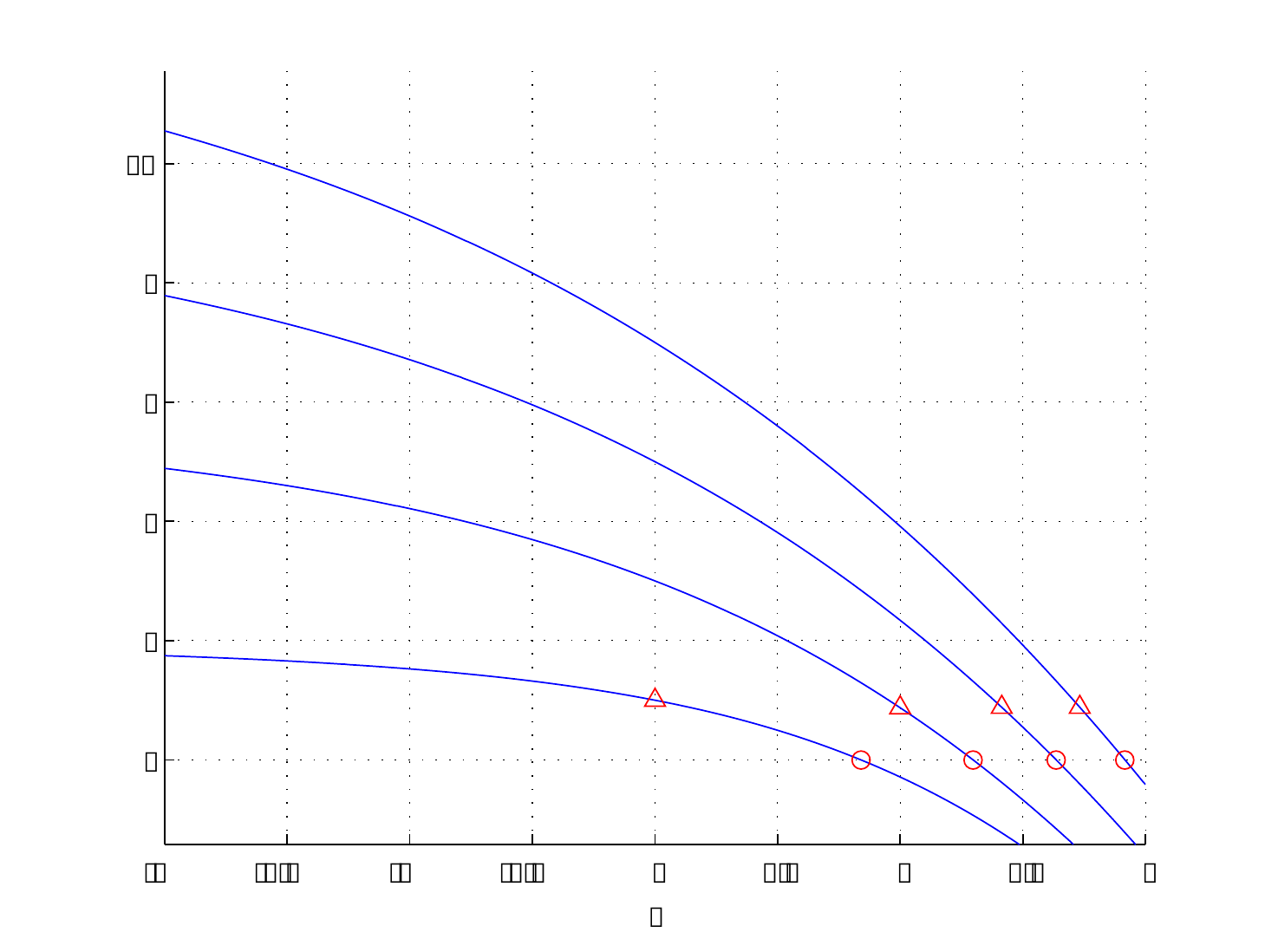}
\caption{Plots of the function $B \mapsto (2n+1) - B F'_{2n+1}(B) /F_{2n+1}(B)$ for $n=0,1,2, 3$. Triangles indicate the points at $\sqrt{n}$. Circles indicate the points at $B^*$.
} \label{plot_B_star}
\end{center}
\end{figure}

By Lemma \ref{lemma_root_n}, for $x > B^* \sqrt{1-t}$ (where $U^*(t,x) = x^{2n+1}$),
\begin{align*} \mathcal{L} U^*(t,x) =(2n+1) \left[  n-\frac {x^2} {1-t}   \right]  x^{2n-1} \leq 0.
\end{align*}
This together with the smooth fit at $B^* \sqrt{1-t}$ (which can be confirmed by simple algebra) verifies the optimality using martingale arguments via It\^o's formula.

\begin{theorem}[Ekstr\"{o}m and Wanntorp \cite{Ekstrom}, Theorems 2.1 and 3.1]
\begin{enumerate}
\item
An optimal stopping time for \eqref{ospbridge_general} is given by $\tau^+(B^*)$ and the value function $U(t,x)$ is given by
\begin{align} \label{expression_U_general}
U(t,x)=U^*(t,x)=\left\{\begin{array}{ll}
(1-t)^{n+1 /2}(B^*)^{2n+1} \frac {F_{2n+1} (x/\sqrt{1-t})} {F_{2n+1}(B^*)}, & \mbox{if $x<B^*\sqrt{1-t}$},\\
x^{2n+1}, &\mbox{if $x\geq B^*\sqrt{1-t}$}.
\end{array}
\right.
\end{align}
\item In particular, when $n=0$, $B^* \simeq 0.84$ is the unique solution to
\begin{align}
\sqrt{2\pi}(1-B^2)e^{B^2/2}\Phi(B)=B.\label{B_equality}
\end{align}
The value function $U(t,x)$ is given by, if $x<B^*\sqrt{1-t}$,
\begin{align} \label{expression_U}
\begin{split}
U(t,x) = U^*(t,x) &= \sqrt{1-t} B^* e^{x^2/(2(1-t)) - (B^*)^2/2}\Phi(x/\sqrt{1-t}) / \Phi(B^*) \\
&=\sqrt{2\pi(1-t)}(1-(B^*)^2)e^{x^2/(2(1-t))}\Phi(x/\sqrt{1-t}),
\end{split}
\end{align}
and it is equal to $x$ otherwise.
\end{enumerate}
\end{theorem}

\subsection{Two-sided exit problem}
Consider now, for fixed integer $q > 0$, the problem of maximizing the absolute value:
\begin{align}
\overline{U}(t, x) := \sup_{t \leq \tau < 1}\e_{t,x} [ |X_{\tau}|^{q}], \quad 0 \leq t < 1, \; x \in \R. \label{ospbridge_general_absolute}
\end{align}
It has been shown by \cite{Ekstrom} that the optimal stopping time is of the form:
\begin{align}
\tau(D):=\inf\{s \geq t: |X_s|  \geq D \sqrt{1-s}  \}, \quad D \geq 0.  \label{def_tau_D}
\end{align}
For $-D \sqrt{1-t} < x < D \sqrt{1-t}$, by  \cite{Ekstrom}, again solving \eqref{eq_ODE} with desired boundary conditions,
\begin{align}
\e_{t,x}[(  1-\tau(D))^{q/2} ] = (1-t)^{q/2} \frac {(F_{q}+G_{q})(x/\sqrt{1-t})} {(F_{q}+G_{q})(D)}, \label{moment_time_q}
\end{align}
and hence
\begin{align}
\e_{t,x} [|X_{\tau(D)}|^q] = D^q \e_{t,x}[(  1-\tau(D))^{q/2} ] = (1-t)^{q/2} D^q \frac {(F_{q}+G_{q})(x/\sqrt{1-t})} {(F_{q}+G_{q})(D)}. \label{moment_x_q}
\end{align}
Here notice that $(F_q+G_q)$ is an even function.

The maximization of this expectation is equivalent to maximizing the function $D \mapsto D^q / {(F_{q}+G_{q})(D)}$, whose derivative equals
\begin{align*}
\frac \partial {\partial D}\frac {D^q} {(F_{q}+G_{q})(D)} = \frac {D^{q-1}} {(F_{q}+G_{q})(D)} \Big[ q- \frac {D (F_{q}+G_{q})'(D)} {(F_{q}+G_{q})(D)} \Big], \quad D > 0.
\end{align*}
Similarly to the arguments above for $B^*$, there exists a maximizer $D^* > 0$, which is a unique root of
\begin{align}
0 = q-D (F_q  + G_q)' (D)/ (F_q + G_q) (D), \label{def_D}
\end{align}
and
\begin{align}
\frac \partial {\partial D}\frac {D^q} {(F_{q}+G_{q})(D)} > 0 \Longleftrightarrow D < D^*, \quad D > 0. \label{monotonicity_D}
\end{align}
We show in Figure \ref{plot_for_absolute_case} the function defined on the right-hand side of \eqref{def_D}.
Similarly to Lemma \ref{lemma_root_n}, we prove the following lower bound for $D^*$.

\begin{lemma} \label{lemma_root_q}We have $D^* \geq \sqrt{(q-1)/2 \vee 0}$.
\end{lemma}
\begin{proof} See Appendix \ref{appendix_proof}.
\end{proof}

\begin{figure}[htbp]
\begin{center}
 \includegraphics[scale=0.65]{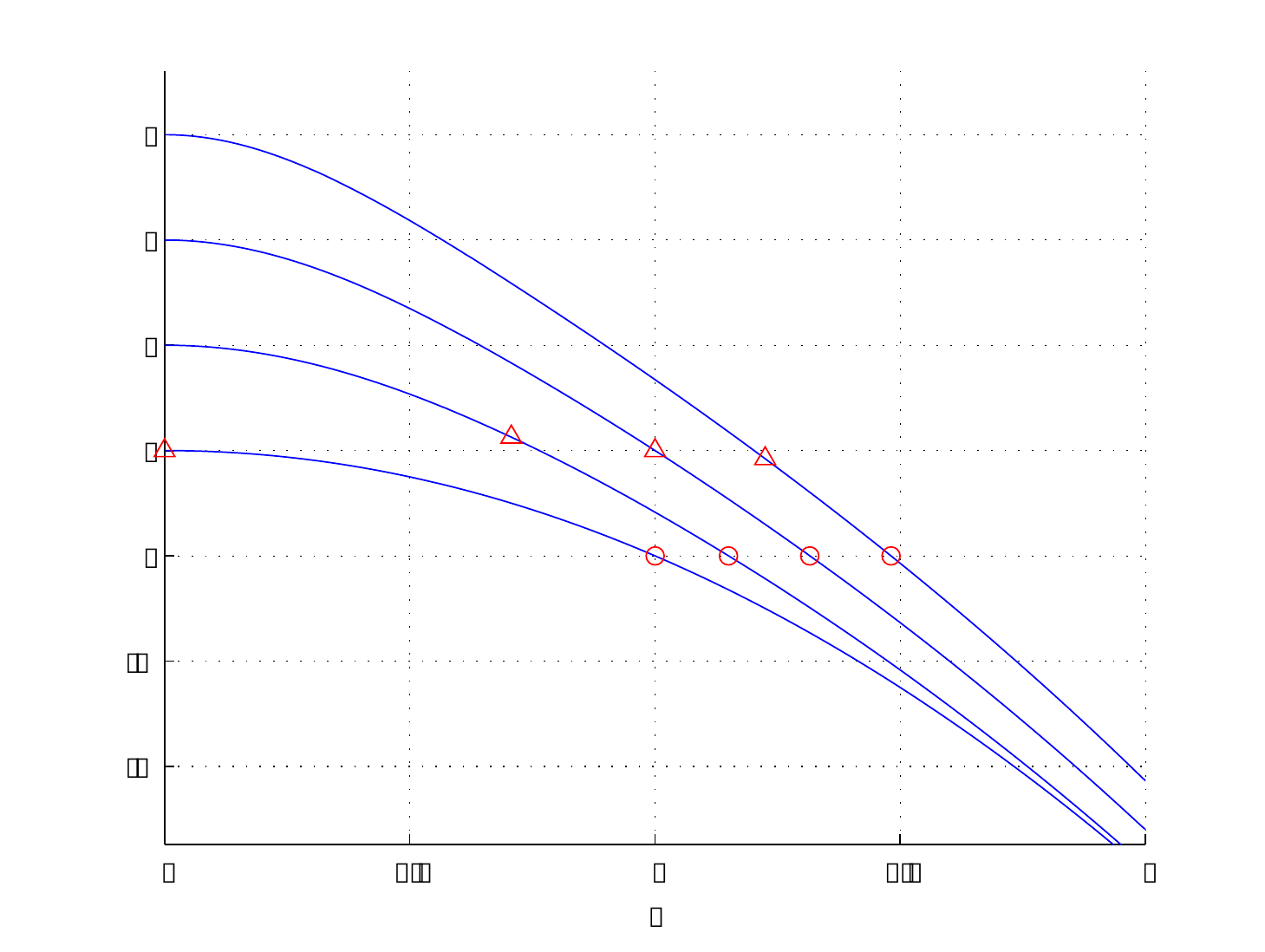}
\caption{Plots of the function $D \mapsto q-D (F_q  + G_q)' (D)/ (F_q + G_q) (D)$ for $q=1,2,3,4$.  The triangles indicate the points at $\sqrt{(q-1)/2}$. Circles indicate the points at $D^*$.} \label{plot_for_absolute_case}
\end{center}
\end{figure}

Define the candidate value function $\overline{U}^*(t, x) := \e_{t,x} [ |X_{\tau(D^*)}|^{q}]$ for $0 \leq t < 1$ and $x \in \R$.  Again, Lemma \ref{lemma_root_q} shows,  for $|x| > D^* \sqrt{1-t}$, that 
\begin{align*}
\mathcal{L} \overline{U}^*(t,x) = q \left[  \frac{q-1} 2 -\frac {|x|^2} {1-t}   \right]  |x|^{q-2} \leq 0.
\end{align*}  This together with the smooth fit at $D^* \sqrt{1-t}$ and $-D^* \sqrt{1-t}$ verifies the optimality.
\begin{theorem}[Ekstr\"{o}m and Wanntorp \cite{Ekstrom}, Theorem 3.2]
An optimal stopping time for \eqref{ospbridge_general_absolute} is given by $\tau(D^*)$ and the value function $\overline{U}(t,x)$ is given by
\begin{align}  \label{U_bar_expression}
\overline{U}(t,x) = \overline{U}^*(t, x) =\left\{\begin{array}{ll}
(1-t)^{q/2} (D^*)^q \frac {(F_{q}+G_{q})(x/\sqrt{1-t})} {(F_{q}+G_{q})(D^*)}, & \mbox{if $|x|< D^*\sqrt{1-t}$},\\
|x|^{q}, &\mbox{if $|x| \geq  D^*\sqrt{1-t}$}.
\end{array}
\right.
\end{align}
\end{theorem}

%\begin{remark}[we can remove this later]
%We have
%\begin{align*}
%F_{k}^{(m)}(y) &= \int_0^\infty u^{k+m-1} \exp \left( yu - \frac {u^2} 2\right) \diff u =  F_{k+m} (y), \quad k \geq 1, m \geq 1.
%\end{align*}
%Hence, by \eqref{F_1_def},
%\begin{align*}
%F_{k}(y) =   \sqrt{2 \pi}  \frac {\partial^{k-1}} {\partial y^{k-1}} \left[ \exp \left(\frac {y^2} 2 \right) \Phi (y) \right], \quad k \geq 1.
%\end{align*}
%Similarly, $G_{k}^{(m)}(y) =  (-1)^{k-1} F_k (-y)$.
%In particular, when $k$ is odd, by \eqref{G_F_1},
%\begin{align*}
%F_{k}(y) +G_{k}(y) = \sqrt{2 \pi} \frac {\partial^{k-1}} {\partial y^{k-1}} \left[ \exp \left(\frac {y^2} 2 \right)  \right].
%\end{align*}
%\end{remark}

\section{Problem 1} \label{section_problem1}
We first solve the optimal double stopping problem of the form:
%Here we will use results from \cite{Ekstrom} to solve the buy-sell problem for a Brownian bridge, i.e.,
\[V(t,x) := \sup_{t \leq \tau_1 \leq \tau_2< 1}\e_{t,x} \left[X_{\tau_2} -  X_{\tau_1} \right], \quad 0 \leq t < 1, \quad x \in \R. \]
%\red{[removed the figure.  We can cite Ivashko and Mazalov if their paper is available.]}
%\red{TO BE REMOVED
%Consider the results from \cite{Ekstrom} and the following slide from a talk by Ivashko and Mazalov based on the discrete case; as far as I know this has not been published yet:
%\begin{figure}[h]
% \includegraphics[scale=0.6]{MazalovKrakow11.pdf}
%\caption{This is the picture from work by Ivashko and Mazalov of the optimal stopping rules they found in the discrete case of a random walk bridge. Note that the boundaries look like discrete versions of constant times $\sqrt{1-t}$.}
%\end{figure}}
First, by the strong Markov property, we can rewrite this as a two-stage problem:
\begin{align}
V(t,x)=  \sup_{t \leq \tau < 1}\e_{t,x} \left[  f(\tau,X_{\tau})\right] \label{V_def}
\end{align}
where
\begin{align*}
f(t,x) := U(t,x) - x, \quad 0 \leq t < 1, \; x \in \R,
\end{align*}
with $U(t,x)$ defined in \eqref{expression_U} as the value function of a single stopping problem.
%\begin{equation}U(t,x):=\sup_{t \leq \tau \leq 1} \e_{t,x} \left[  X_{\tau}\right],\label{ospbridge}\end{equation}
%defined as a special case of \eqref{ospbridge_general} with $n=0$.

It is expected that the first optimal stopping time is of the form
\begin{align}
\tau^-(C):=\inf\{s \geq t: X_s  \leq C \sqrt{1-s} \}, \label{lower_crossing_time}
%\tau_2^*&=\tau^+(B):=\inf\{\tau^* \leq u \leq 1: X_u \geq B\sqrt{1-u}\}
\end{align}
for some $C \in \R$.  The corresponding second stopping time becomes  $\inf\{s \geq \tau^-(C) : X_s  \geq  B^* \sqrt{1-s}  \}$.

For $C \in \R$ and $x>C\sqrt{1-t}$, by \eqref{F_1_def}, \eqref{expectation_moments_2n_1} and symmetry,
\begin{multline*}
\e_{t,x} \left[C \sqrt{1-\tau^-(C)}\right] = \e_{t,x} \left[  X_{\tau^-(C)}\right]
=-\e_{t,-x} \left[  X_{\tau^+(-C)}\right]
\\
=\frac{C}{\Phi(-C)}e^{-C^2/2}\sqrt{1-t}\Phi\left(-{x} /{\sqrt{1-t}}\right)e^{x^2/(2(1-t))}.
\end{multline*}

Now we focus on the function, for $C \leq B^*$, 
\begin{align}
V_C(t,x) := \e_{t,x} \left[  f(\tau^-(C),X_{\tau^-(C)})\right], \quad 0 \leq t < 1, \; x \in \R. \label{V_C_def}
\end{align}
This becomes $f(t,x)$ for $x\leq C\sqrt{1-t}$ whereas for $x>C\sqrt{1-t}$, by \eqref{expression_U},
\begin{eqnarray*}
V_C(t,x)
&=&\e_{t,x} \left[  \sqrt{2\pi(1-\tau^-(C))}(1-(B^*)^2)e^{C^2/2}\Phi(C)  -  X_{\tau^-(C)}\right]\\
%&=&\e_{t,x} \left[  \sqrt{2\pi(1-\tau^-(C))}(1-B^2)e^{C^2/2}\Phi(C)  -  X_{\tau^-(C)}\right]\\
&=& \left( \sqrt{2\pi}(1-(B^*)^2)e^{C^2/2}\frac{\Phi(C)}{C} -1\right)\e_{t,x}\left[ C \sqrt{1-\tau^-(C)} \right]\\
%&=&\left( \sqrt{2\pi}(1-(B^*)^2)e^{C^2/2}\frac{\Phi(C)}{C} -1\right)\frac{C}{\Phi(-C)}e^{-C^2/2}\sqrt{1-t}\Phi\left(-\frac{x}{\sqrt{1-t}}\right)e^{x^2/(2(1-t))} \\
&=&v(C) \sqrt{1-t}\sqrt{2 \pi} \Phi\left(-{x}/{\sqrt{1-t}}\right)e^{x^2/(2(1-t))},
\end{eqnarray*}
where we define
\begin{align}
v(C):=  \frac 1 {\Phi(-C)} \Big[ (1-(B^*)^2){\Phi(C)} -C\frac {e^{-C^2/2}} {\sqrt{2 \pi}} \Big], \quad C \leq B^*.  \label{v_single}
\end{align}

The idea now is to identify $C$ that maximizes $V_C(\cdot,\cdot)$ (or equivalently $v(\cdot)$) and then use a verification lemma to show the optimality of the corresponding strategy.  
Hence, we consider optimizing the function $v(C)$ on $(-\infty,B^*]$.  Figure \ref{figure_plot_C} shows a plot of this function (using the definition of $B^*$ above).  
It is remarked that only the maximality of $v(C)$ and $V_C$ over $(-\infty, B^*]$ is needed; Lemma \ref{lemma_maximality_c_star} below is used for the proof of Lemma  \ref{lemma_domination_prob1}, where only the maximality over $(-\infty, B^*]$ is necessary.
It is clearly suboptimal to choose $C>B^*$ as the corresponding strategies would lead to an expected pay-off of zero at or above the boundary $x=C\sqrt{1-t} > B^* \sqrt{1-t}$, with the first and second stoppings happening at the same time.
\begin{figure}[htbp]
 \includegraphics[scale=0.6]{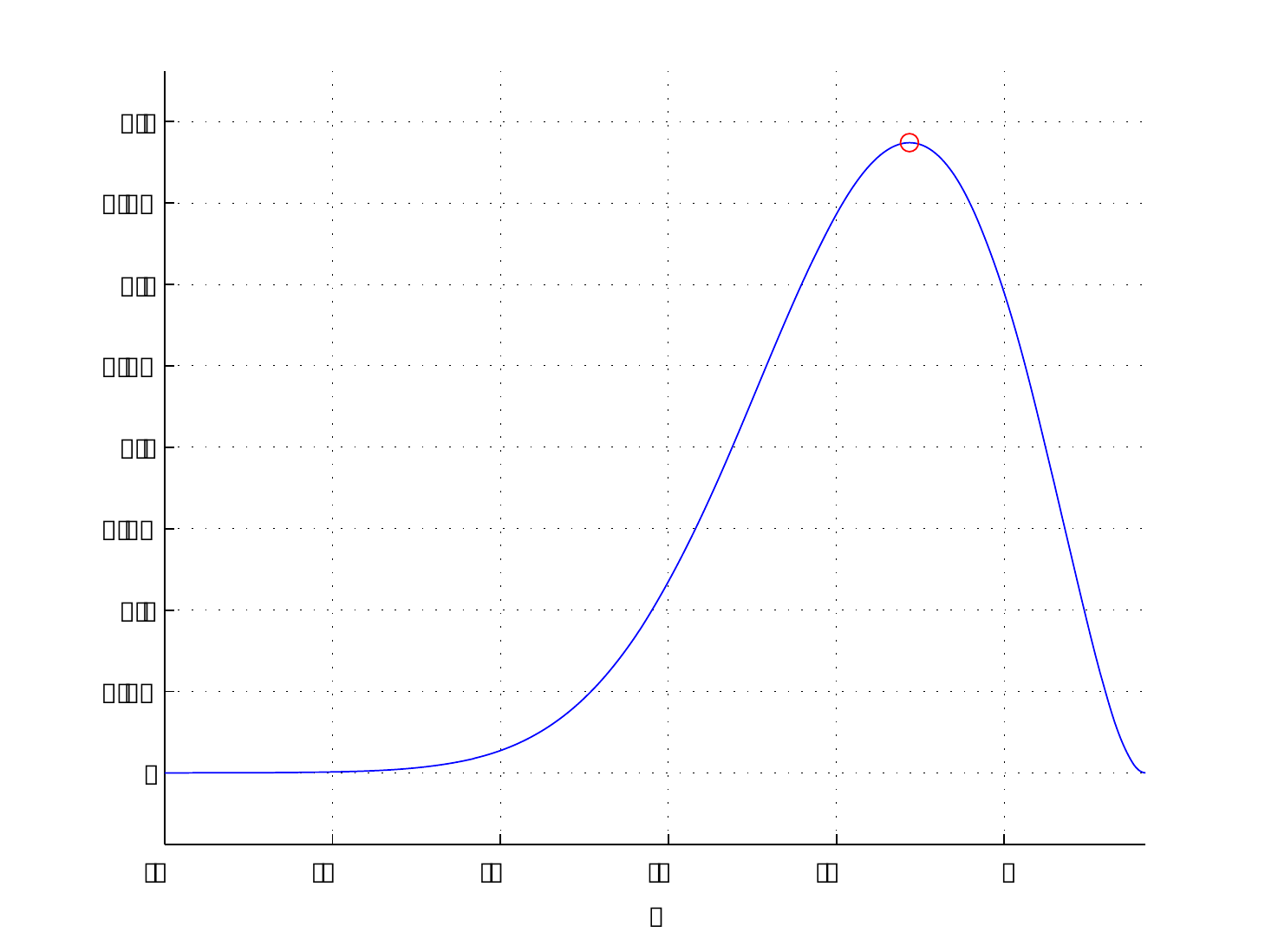}
\caption{The function $v$ on $[-5, B^*]$.  The graph indicates that it has a unique maximum  $C^* \simeq-0.564$ (the point indicated by the circle); note that $|C^*|$ is smaller than $B^*\simeq 0.84$.} \label{figure_plot_C}
\end{figure}

\begin{lemma} \label{lemma_maximality_c_star}There exists a unique  $C^* < 0$ that maximizes $v(\cdot)$ over $(-\infty,B^*]$ such that $u(C^*) = 0$ where we define
\[u(C) :=1-(B^*)^2-(1-C^2)\Phi(-C)-\frac{C}{\sqrt{2\pi}}e^{-C^2/2}, \quad C \leq B^*. \]
\end{lemma}
\begin{proof}%[Proof of Lemma \ref{lemma_maximality_c_star}]
For all $C \leq B^*$, using $\Phi(C) + \Phi(-C) =1$, 
\begin{align} \label{f_derivative_h}
\begin{split}
v'(C)
%&=\frac{e^{-C^2/2}}{\sqrt{2 \pi} (\Phi(-C))^2}\left(\sqrt{2\pi}(1-(B^*)^2)\frac{\Phi(C)+\Phi(-C)}{\sqrt{2\pi}}+C^2\Phi(-C)-\Phi(-C)-\frac{C}{\sqrt{2\pi}}e^{-C^2/2}\right)\\
&=\frac{e^{-C^2/2}}{\sqrt{2 \pi} (\Phi(-C))^2}u(C),
\end{split}
\end{align}
and
\begin{eqnarray*}
u'(C)=
2C\Phi(-C)-\frac{C^2-1}{\sqrt{2\pi}}e^{-C^2/2}-\frac{1}{\sqrt{2\pi}}e^{-C^2/2}+\frac{C^2}{\sqrt{2\pi}}e^{-C^2/2}
= 2C\Phi(-C).
\end{eqnarray*}
On $(-\infty, 0)$, $u'(C)$ is uniformly negative.  Moreover, $u(-\infty)=\infty$ and $u(0)=1/2-(B^*)^2<0$. Thanks to the continuity of $u$, this implies that there is a unique solution to the equation $u(C)=0$ on $(-\infty,0)$, which we call $C^*$.

It remains to show that $C^*$ indeed maximizes the function $v$ over $(-\infty, B^*]$. From \eqref{f_derivative_h}, we see that $v'(C)$ and $u(C)$ are of the same sign. Hence, on $(-\infty,0)$, $v$ is strictly increasing on $(-\infty,C^*)$ and is strictly decreasing on $(C^*,0)$ showing that $C^*$ is the unique maximizer on $(-\infty,0]$.

We now extend this result to the domain $(0, B^*]$.
On $(0,B^*]$, $u'$ is uniformly positive and  hence $u$ is monotonically increasing.  Hence, $v$ will be strictly increasing on $(z,\infty)$ as soon as $v'(z)=0$ (or $u(z)=0$) for some $z>0$.  This means that $v$ on $(0, B^*]$ is dominated by the maximum value of $v(0)$ and $v(B^*)$. Finally,  observe that $v(B^*)=0$ (by how $B^*$ is chosen as in \eqref{B_equality}), which is smaller than $v(C^*)$. This completes the proof.
\end{proof}

Now we define our candidate value function, for  $0 \leq t < 1$ and $x \in \R$,
\begin{align}
V^*(t,x) := V_{C^*}(t,x) = \left\{ \begin{array}{ll} \sqrt{1-t} \sqrt{2 \pi}\Phi\left(-{x}/ {\sqrt{1-t}}\right)e^{x^2/(2(1-t))} v(C^*), & x > C^* \sqrt{1-t} \\ f(t,x), & x \leq C^* \sqrt{1-t} \end{array} \right\}. \label{def_candidate}
\end{align}
We plot, in Figure \ref{plot_prob1}, the functions $V^*$ and $f$ for fixed $t = 0$ as a function of $x$; this suggests Lemmas \ref{lemma_domination_prob1} and \ref{lemma_smooth_fit}, which we shall prove analytically below. 

\begin{figure}[htbp]
\begin{center}
 \includegraphics[scale=0.65]{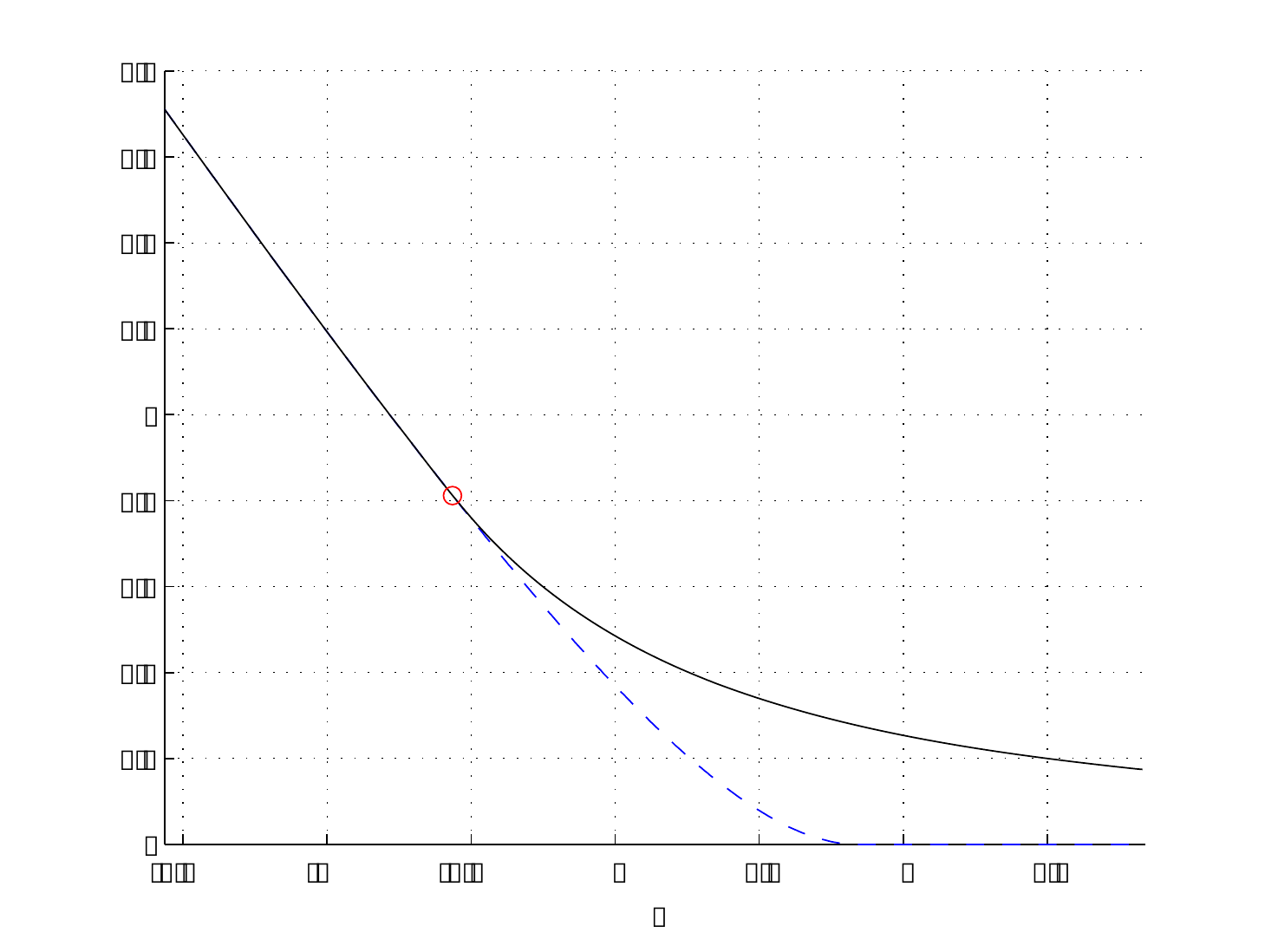}
\caption{Plots of $V^*(0,\cdot)$ (solid) and $f(0,\cdot)$ (dotted). The circle indicates the point at $C^*$.} \label{plot_prob1}
\end{center}
\end{figure}

\begin{lemma} \label{lemma_domination_prob1}
We have $V^*(t,x)\geq f(t,x)$, for any $0 \leq t < 1$ and $x \in \R$.
\end{lemma}
\begin{proof}%[Proof of Lemma \ref{lemma_domination_prob1}]
To derive the inequality, we remark that we only need to consider $x< B^*\sqrt{1-t}$. Indeed, for $x\geq B^*\sqrt{1-t}$ it holds that $V^*(t,x)\geq 0=f(t,x)$.

Consider $C^*\sqrt{1-t}<x<B^* \sqrt{1-t}$. Due to continuous fit and the maximality of $C^*$ on $(-\infty,B^*]$ as in Lemma \ref{lemma_maximality_c_star}, we derive that
\[V^*(t,x)=V_{C^*}(t,x)\geq V_{x/\sqrt{1-t}}(t,x)=f(t,x).\]
Finally,  for $x \leq C^*\sqrt{1-t}$, we have $V^*(t,x)= f(t,x)$.
\end{proof}

%\begin{remark}\red{[TO BE REMOVED]}
% Of course this is by no means the end of the story as at the moment our choice of strategies is just a guess and should be checked using a verification lemma. At least Mathematica again seems to show that we do have smooth fit at the boundary $C\sqrt{1-t}$. The following is a picture of
%\[\frac{\partial}{\partial x}V(t,x)-\frac{\partial}{\partial x}U(t,x)+1\]
%at the boundary $x=C\sqrt{1-t}$. Note that the pay-off function is $U(t,x)-x$ hence the above should be zero at the boundary:
%\begin{figure}[h]
% \includegraphics[scale=0.6]{smoothfit}
%\caption{Difference between the partial derivatives (with respect to $x$) of the pay-off and value function at the boundary $x=C\sqrt{1-t}$. }
%\end{figure}
%Seems to be only rounding errors so this looks promising.
%Hopefully the above analysis of the function $F$ will allow us to get $V(t,x)$ through a verification lemma.
%\end{remark}

%\begin{remark}[TO BE MOVED]
%If we have checked the solution it is pretty straightforward to rescale it and get the solution when considering a Brownian bridge on a small time interval, the constants $C^*$ and $B^*$ would remain the same but the value function is rescaled and $\sqrt{1-t}$
%\end{remark}

Before we verify the optimality, we shall prove the smoothness so as to use It\^o's formula.  In view of \eqref{def_candidate}, $V^*(t,x)$ is twice-differentiable in $x$ at any $(t,x)$ such that $x \neq C^* \sqrt{1-t}$.  Hence the smoothness on  $x = C^* \sqrt{1-t}$ is our only concern.

\begin{lemma}  \label{lemma_smooth_fit} We have smooth fit:
\begin{align}
\lim_{x \downarrow C^* \sqrt{1-t}}\frac \partial {\partial x} V^*(t,x) = \lim_{x \uparrow C^* \sqrt{1-t}} \frac \partial {\partial x} f(t,x), \quad 0 \leq t < 1.  \label{smooth_fit_condition}
\end{align}
\end{lemma}
\begin{proof} See Appendix \ref{appendix_proof}.
\end{proof}
Note that $V^*$ is continuously differentiable in $t$ for any $(t,x)$ such that  $x \neq C^* \sqrt{1-t}$; the differentiability on  $x = C^* \sqrt{1-t}$ in $t$ can be shown by slightly modifying the proof of Lemma \ref{lemma_smooth_fit}.

\begin{lemma} \label{generator_prob1} (i) For $(t,x)$ such that $x > C^* \sqrt{1-t}$, we have $\mathcal{L} V^* (t,x)  = 0$.
(ii) For $(t,x)$ such that $x < C^* \sqrt{1-t}$,  we have $\mathcal{L} V^* (t,x) \leq 0$.
\end{lemma}
\begin{proof}
(i) It is clear by the construction of the expected value as a solution to the ODE \eqref{eq_ODE}.

(ii)  As $V^*(t,x) = U(t,x) -x$ and because $\mathcal{L} U (t,x) = 0$ for $x <  C^* \sqrt{1-t} < B^* \sqrt {1-t}$ in view of  \eqref{expression_U},
\begin{align*}
\mathcal{L} V^* (t,x) =  \mathcal{L} U (t,x)  + \frac x {1-t}  =  \frac x {1-t},
\end{align*}
which is negative as $x < C^* \sqrt{1-t} < 0$.
\end{proof}

We now have the main result of this section.
\begin{theorem}  \label{theorem_problem1}The function
$V^*$ as defined in \eqref{def_candidate} is the value function.  Namely, $V(t, x) = V^*(t,x)$ for every $0 \leq t < 1$ and $x \in \R$; optimal stopping times are
\begin{align*}
\tau_1^* := \tau^-(C^*) \quad \textrm{and} \quad \tau_2^* := \inf\{ s \geq \tau^-(C^*): X_s\geq  B^* \sqrt{1-s}\}.
\end{align*}
\end{theorem}
\begin{proof}%[Proof of Theorem \ref{theorem_problem1}]
%By applying this in \eqref{def_candidate} for $x > C^* \sqrt{1-t}$,
%\begin{align*}
%\mathcal{L} V^* (t,x)  = 0, \quad x > C^* \sqrt{1-t}.
%\end{align*}
Thanks to the smooth fit as in Lemma \ref{lemma_smooth_fit}, It\^o's formula applies and, for all $(s,X_s)$ such that $ X_s \neq C^* \sqrt{1-s}$,
\begin{align}
\diff V^*(s, X_s) &= \mathcal{L} V^*(s, X_s) \diff s + \frac \partial {\partial x} V^*(s, X_s) \diff W_s \leq \frac \partial {\partial x} V^*(s, X_s) \diff W_s, \label{V_diff__eqn}
\end{align}
where the inequality holds by Lemma  \ref{generator_prob1}.

%Hence, for all $0 \leq t \leq u \leq 1$,
%\begin{align*}
%V^*(u, X_u)   - \int_t^u \frac {X_s} {1-s} 1_{\{  X_s < C^* \sqrt{1-s} \}} \diff s
%&= V^*(t, X_t) + \int_t^u  \frac \partial {\partial x} V^*(s, X_s)  1_{\{ X_s \neq C^* \sqrt{1-s} \}} \diff W_s.
%\end{align*}

In the problem \eqref{V_def}, because stopping at or above $B^* \sqrt{1-t}$ attains a zero payoff, which is clearly suboptimal, we can focus on stopping times $\nu$ such that $X_\nu < B^* \sqrt{1-\nu}$ (and hence $f(\nu, X_{\nu}) > 0$) a.s. 
For any such $[t,1)$-valued stopping time $\nu$, with $\tau(M)$ as defined in \eqref{def_tau_D} for $M > B^*$, we have
\begin{align*}
\e_{t,x} [f(\nu \wedge \tau(M), X_{\nu \wedge \tau(M)})]  \leq \e_{t,x} [V^*(\nu \wedge \tau(M), X_{\nu \wedge \tau(M)}) ] \leq  V^*(t, x),
\end{align*}
where the first and second inequalities hold by Lemma \ref{lemma_domination_prob1} and \eqref{V_diff__eqn}, respectively.

In order to take $M \rightarrow \infty$, we decompose the left-hand side as
\begin{align*}
\e_{t,x} [f(\nu \wedge \tau(M), X_{\nu \wedge \tau(M)})]  = \e_{t,x} [f(\nu, X_{\nu}) 1_{\{\nu < \tau(M)\}}]   + \e_{t,x} [f(\tau(M), X_{\tau(M)}) 1_{\{\nu \geq \tau(M)\}}].
\end{align*}
The first expectation of the right-hand side converges via monotone convergence to $\e_{t,x} [f(\nu, X_{\nu})]$ because $f$ is nonnegative. On the other hand,  $f(\tau(M), X_{\tau(M)}) 1_{\{\nu \geq \tau(M)\}}$ is uniformly integrable for $\{ \tau(M), M > B^*\}$.  Indeed, $f(s,y) = 0$ for $y \geq B^* \sqrt{1-s}$.  In addition, for $y < B^* \sqrt{1-s}$, the first equation of \eqref{expression_U} gives that $U(s,y) \leq \sqrt{1-s} B^*$, and hence we have a bound
\begin{align*}
|f(s,y) |= U(s,y) - y \leq \sqrt{1-s} B^* + |y|, \quad y < B^* \sqrt{1-s}.
\end{align*}
Therefore,
\begin{align*}
|f(\tau(M), X_{\tau(M)})| 1_{\{\nu \geq \tau(M)\}}
\leq \sqrt{1-\tau(M)} B^* + |X_{\tau(M)}|,
\end{align*}
which is uniformly integrable in view of \eqref{moment_x_q} (which is maximized by setting $D = D^*$).  Now as $M \rightarrow \infty$, because $\tau(M) \rightarrow 1$ and $X_{\tau(M)} \rightarrow 0$ a.s., we have 
\begin{align*}
\e_{t,x} [f(\tau(M), X_{\tau(M)}) 1_{\{\nu \geq \tau(M)\}}] \rightarrow 0.
\end{align*}
 In sum, we have  $\e_{t,x} [f(\nu, X_{\nu})]  \leq V^*(t,x)$.

This together with the fact $V^*$ is attained by an admissible stopping time $\tau^-(C^*) \in [t,1)$ shows the result.
\end{proof}

\section{Problem 2}  \label{section_problem2}
We now consider the problem, for given  integer $n \geq 0$,
\begin{multline*}
J(t,x) := \sup_{t \leq \tau_1 \leq \tau_2 <1}\e_{t,x}[(X_{\tau_2}^{2n+1} -  X_{\tau_1}^{2n+1}) 1_{\{X_{\tau_1} \leq 0\}} + (X_{\tau_1}^{2n+1} -  X_{\tau_2}^{2n+1}) 1_{\{X_{\tau_1} > 0\}}],  \\  0 \leq t<  1, \; x \in \R.
\end{multline*}
By the strong Markov property, we can  rewrite it as
\begin{align*}
J(t,x) 
%&= \sup_{t \leq \tau_1 \leq 1}\e_{t,x}[(\sup_{\tau_2 \geq \tau_1}X_{\tau_2}^{2n+1} -  X_{\tau_1}^{2n+1}) 1_{\{X_{\tau_1} \leq 0 \}} + (-\inf_{\tau_2 \geq \tau_1}X_{\tau_2}^{2n+1}+  X_{\tau_1}^{2n+1}) 1_{\{X_{\tau_1} > 0 \}}] \\
&= \sup_{t \leq \tau < 1}\e_{t,x}[g(\tau, X_{\tau})],
\end{align*}
where
\begin{align*}
g(t,x) := (U(t, x) - x^{2n+1}) 1_{\{x \leq 0 \}} + (U(t, -x) +  x^{2n+1}) 1_{\{x > 0 \}}, \quad 0 \leq t < 1, \; x \in \R.
\end{align*}

By the symmetry of $g$ with respect to $x$, we expect, for some $D \geq 0$, that the first stopping time has a form $\tau(D)$ defined as in \eqref{def_tau_D};
 the second stopping time becomes
\begin{align*}
\left\{ \begin{array}{ll} \inf\{s\geq \tau(D): X_s\geq B^*\sqrt{1-s} \}, & \textrm{if } X_{\tau(D)} \leq -D \sqrt {1-\tau(D)}, \\
\inf\{s\geq \tau(D): X_s\leq -B^* \sqrt{1-s}\}, & \textrm{if } X_{\tau(D)} \geq D \sqrt {1-\tau(D)}.
\end{array} \right.   
\end{align*}
%  if $X_{\tau(D)} = -D \sqrt {1-\tau(D)}$
%\[\inf\{t\geq \tau(D): X_t\geq B^*\sqrt{1-t}\}\]
%whereas if $X_{\tau(D)} = D \sqrt {1-\tau(D)}$
%\[\inf\{t\geq \tau(D): X_t\leq -B^* \sqrt{1-t}\}.\]
Define the corresponding payoff by
\begin{align*}
J_D(t,x)
&:= \e_{t,x}[g(\tau(D), X_{\tau(D)})], \quad 0 \leq t < 1, \; x \in \R.
\end{align*}

We first rewrite it as a function of $F_{2n+1}$ and $G_{2n+1}$ as defined in \eqref{def_F_G}.
\begin{lemma} \label{lemma_J_D}Given $D \geq 0$, we have, for all $0 \leq t < 1$ and $- D \sqrt{1-t} \leq x \leq D \sqrt{1-t}$,
\begin{align}
J_D(t,x)
&= (1-t)^{n+1 /2} {(F_{2n+1}+G_{2n+1})\big(x  /{\sqrt{1-t}} \big)}j(D),  \label{J_D_expression}
\end{align}
where
\begin{align*}
j(D) := \frac 1 {(F_{2n+1}+G_{2n+1})(D)}\left[ D^{2n+1} + (B^*)^{2n+1}\frac {G_{2n+1} (D)} {F_{2n+1}(B^*)} \right].
\end{align*}
\end{lemma}
\begin{proof}%[Proof of Lemma \ref{lemma_J_D}] 
Under the initial condition $- D \sqrt{1-t} \leq x \leq D \sqrt{1-t}$, we have $\p_{t,x}$-a.s.
\begin{align*}
&g(\tau(D),X_{\tau(D)}) \\ &= (U(\tau(D),X_{\tau(D)}) - X_{\tau(D)}^{2n+1}) 1_{\{ X_{\tau(D)} \leq 0 \}} + (U(\tau(D),-X_{\tau(D)}) +  X_{\tau(D)}^{2n+1}) 1_{\{X_{\tau(D)} > 0 \}} \\
&= (U(\tau(D),-D \sqrt{1-\tau(D)}) + D^{2n+1} (1-\tau(D))^{n+1 / 2}) 1_{\{ X_{\tau(D)} \leq 0 \}} \\ &\qquad + (U(\tau(D),-D \sqrt{1-\tau(D)}) + D^{2n+1} (1-\tau(D))^{n+1/ 2}) 1_{\{X_{\tau(D)} > 0 \}} \\
&= U(\tau(D),-D \sqrt{1-\tau(D)}) + D^{2n+1} (1-\tau(D))^{n+1 /2}.
\end{align*}
Hence, we can write, by \eqref{expression_U_general},
\begin{align*}
J_D(t,x)
&= \e_{t,x}[ U(\tau(D),  -D \sqrt {1-\tau(D)}) + D^{2n+1} (1-\tau(D))^{n+1 /2}]  \\
%&= \e_{t,x}\Big[  D^{2n+1} (1-\tau(D))^{n+1 /2}+ (1-\tau(D))^{n+1 /2}(B^*)^{2n+1}\frac {F_{2n+1} (-D)} {F_{2n+1}(B^*)}\Big]  \\
&= \e_{t,x}[(  1-\tau(D))^{n+1 /2} ] \left[D^{2n+1} + (B^*)^{2n+1}\frac {F_{2n+1} (-D)} {F_{2n+1}(B^*)} \right].
\end{align*}
The proof is now complete by \eqref{moment_time_q}.
%
%
%
%Hence
%\begin{align*}
%J_D(t,x)
%&= (1-t)^{n+\frac 1 2} {(G_{2n+1}+F_{2n+1})(x/\sqrt{1-t})}G(D).
%\end{align*}
\end{proof}

\begin{figure}[htbp]
\begin{center}
%\begin{minipage}{1.0\textwidth}
%\centering
%\begin{tabular}{cc}
% \includegraphics[scale=0.45]{figure1}  &  \includegraphics[scale=0.45]{figure3}\\
%$n=0: (B^* = 0.8399 > 0 = \sqrt{0})$ & $n=1: (B^* = 1.2972 > 1 = \sqrt{1})$\vspace{0.3cm} \\
% \includegraphics[scale=0.45]{figure5} &  \includegraphics[scale=0.45]{figure15} \\
%$n=2: (B^* = 1.6360 > 1.414 = \sqrt{2})$ & $n=7: (B^* = 2.7695 > 2.6457 = \sqrt{7})$\vspace{0.3cm} \\
%\end{tabular}
%\end{minipage}
 \includegraphics[scale=0.65]{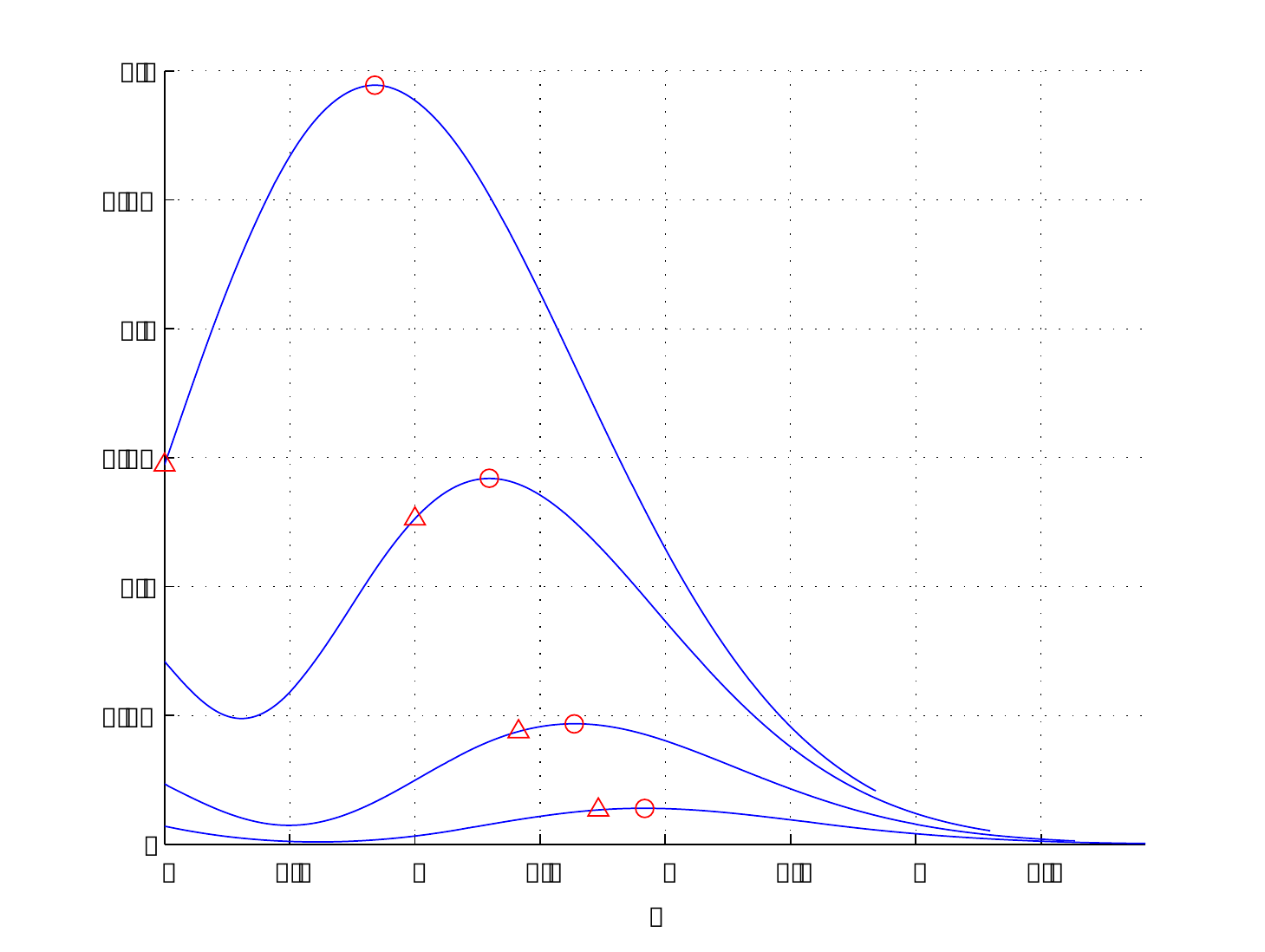}
\caption{Plots of $j$ for $n=0,1,2,3$. Circles indicate the points at $B^*$ and triangles indicate the points at $\sqrt{n}$.} \label{plot_G}
\end{center}
\end{figure}

In view of \eqref{J_D_expression}, we want to maximize the function $j$.  It turns out that it is maximized by $B^*$ as in \eqref{B_star}.  See Figure \ref{plot_G} for a numerical plot of this function.

\begin{lemma}\label{lemma_maximizer_problem2}
For any $n \geq 0$, $B^*$ maximizes $j$.
\end{lemma}
\begin{proof}%[Proof of Lemma \ref{lemma_maximizer_problem2}] 
For any $D \geq 0$,
\begin{align*}
j(D)
%&= G(B^*) + \frac {(B^*)^{2n+1}} {(G_{2n+1}+F_{2n+1})(D)}\left[ \frac {F_{2n+1}(B^*)-F_{2n+1} (D)} {F_{2n+1}(B^*)} \right] +\frac { D^{2n+1} -  (B^*)^{2n+1}} {(G_{2n+1}+F_{2n+1})(D)} \\
&= j(B^*) +\frac {D^{2n+1}} {(F_{2n+1}+G_{2n+1})(D)}  - \frac {(B^*)^{2n+1}} {(F_{2n+1}+G_{2n+1})(D)}\frac {F_{2n+1} (D)} {F_{2n+1}(B^*)} 
\\
&= j(B^*)
+\frac {D^{2n+1}} {(F_{2n+1}+G_{2n+1})(D)} \Big[ 1  -\frac {(B^*)^{2n+1}} {D^{2n+1}} \frac {F_{2n+1} (D)} {F_{2n+1}(B^*)}  \Big].
\end{align*}
Because  $B^*$ is the maximizer of $B \mapsto {B^{2n+1}} / {F_{2n+1}(B)}$ (see \eqref{monotonicity_B}) and $F_{2n+1}$ is nonnegative, we have
\begin{align*}
1  -\frac {(B^*)^{2n+1}} {D^{2n+1}} \frac {F_{2n+1} (D)} {F_{2n+1}(B^*)}   \leq 0,
\end{align*}
which shows $j(D) \leq j(B^*)$, as desired.
\end{proof}

%\begin{remark}
%Figure \ref{plot_G} shows the function $G$.  We can observe, as in the case $n=0$ that it is maximized when $D=B^*$, and also satisfies $B^* \geq \sqrt{n}$, \red{which will be used in the verification later}.  I don't know if we can show this analytically, but we can at least analytically show that the $B^*$ attains a local maximum.
%\end{remark}

\begin{remark} \label{remark_j} We have
 $j(B^*) = {(B^*)^{2n+1}}/{F_{2n+1}(B^*)}$.

\end{remark}

Now setting $D = B^*$, we have our candidate value function
\begin{align}
J^*(t,x) := J_{B^*}(t,x)
&= \left\{ \begin{array}{ll} (1-t)^{n+1 /2} {(F_{2n+1}+G_{2n+1})\big(x / {\sqrt{1-t}}\big)}j(B^*), & |x| < B^* \sqrt{1-t}, \\ g(t,x), & |x| \geq B^* \sqrt{1-t}.  \end{array}\right. \label{def_candidate2}
\end{align}
Figure \ref{plot_prob2} plots $J^*$ and $g$ for $t=0$ as a function of $x$; this suggests Lemmas \ref{lemma_domination_prob2} and \ref{smoothness_problem2}, which we shall prove analytically below. 

\begin{figure}[htbp]
\begin{minipage}{1.0\textwidth}
\centering
\begin{tabular}{cc}
 \includegraphics[scale=0.56]{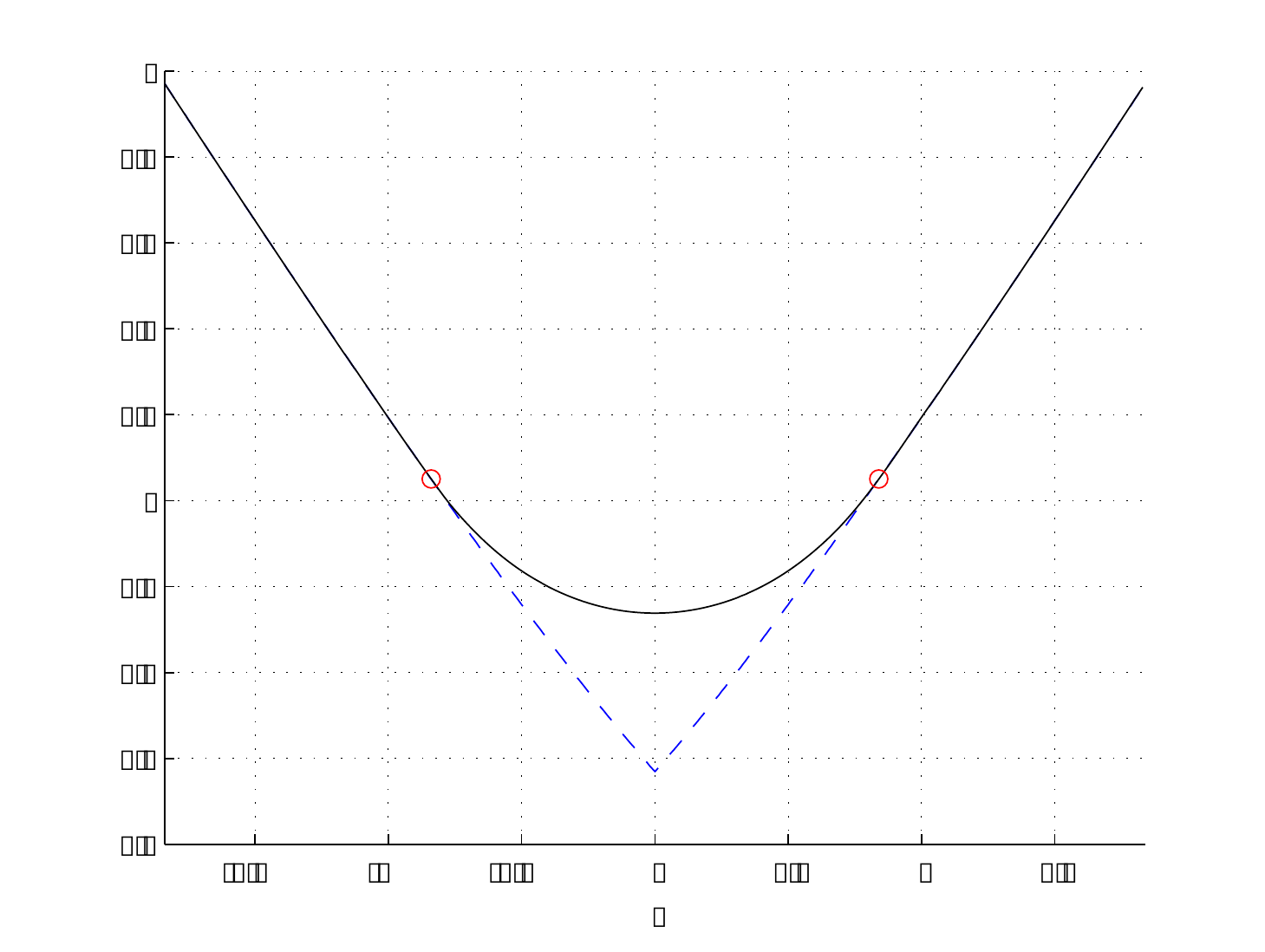} & \includegraphics[scale=0.56]{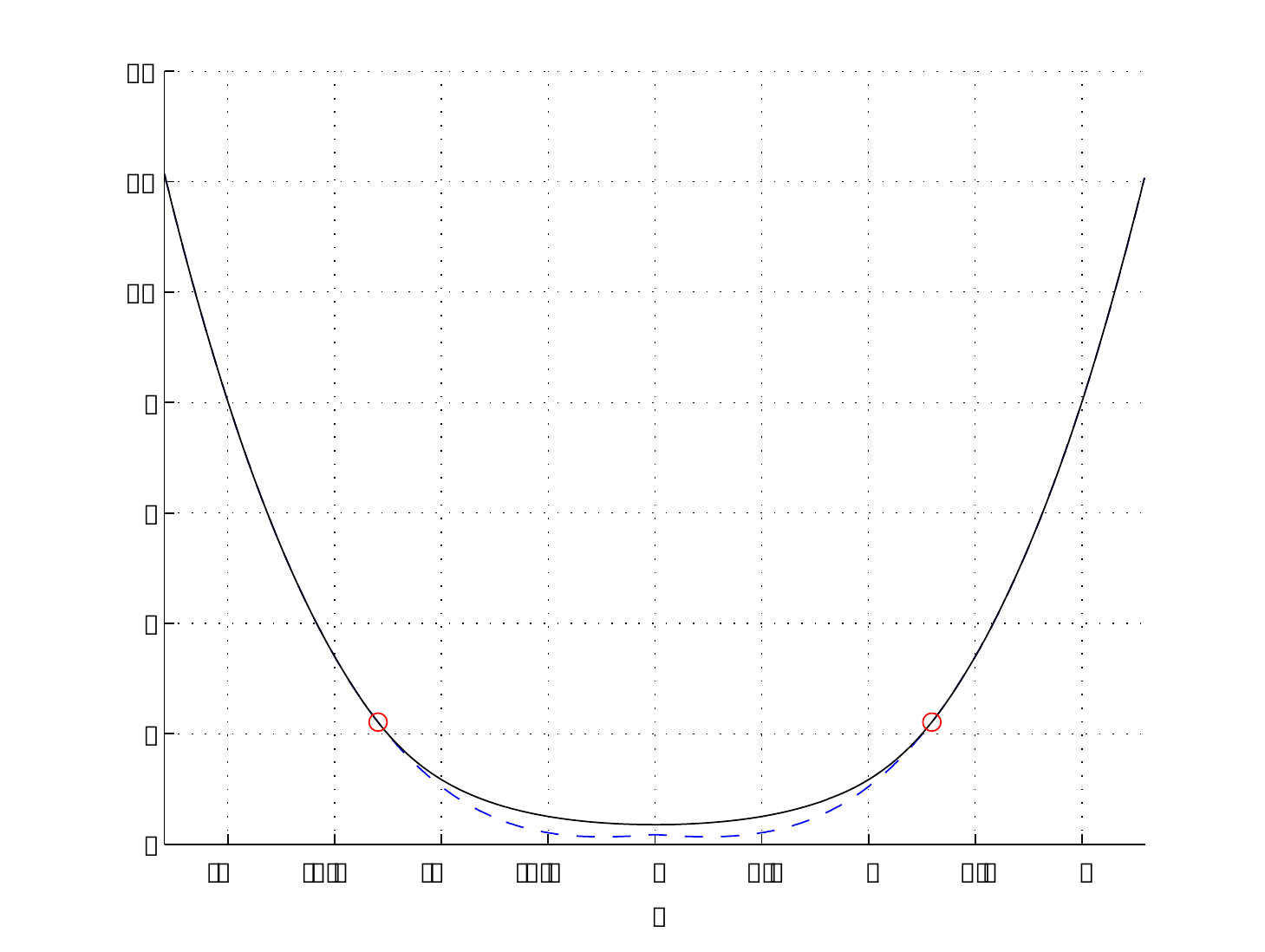}
\end{tabular}
\end{minipage}\caption{Plots of $J^*(0,\cdot)$ (solid) and $g(0,\cdot)$ (dotted) for $n=0$ (left) and $n=1$ (right). Circles indicate the points at $B^*$ and $-B^*$.} \label{plot_prob2}
\end{figure}

\begin{lemma} \label{lemma_domination_prob2}
We have $J^*(t,x)\geq g(t,x)$, for any $0 \leq t < 1$ and $x \in \R$.
\end{lemma}
\begin{proof}
Suppose $0 < x<B^* \sqrt{1-t}$.  Due to continuous fit and the maximality of $B^*$ on $[0,\infty)$ we derive that
\[J^*(t,x)=J_{B^*}(t,x)\geq J_{x/\sqrt{1-t}}(t,x)= g(t,x).\]
Suppose $0 \geq  x > -B^* \sqrt{1-t}$.   By the symmetry of $J^*$ and $g$ with respect to $x$,
\[J^*(t,x)=J^*(t,|x|) \geq g(t, |x|) = g(t,x). \]
Finally,  for $|x| \geq B^*\sqrt{1-t}$, we have $J^*(t,x)= g(t,x)$.
\end{proof}

 In view of \eqref{def_candidate2}, $J^*(t,x)$ is twice-differentiable in $x$ at any $(t,x)$ such that $|x| \neq B^* \sqrt{1-t}$.  As we shall show next, on  $|x| = B^* \sqrt{1-t}$, differentiability holds (the differentiability with respect to $t$ holds similarly).

\begin{lemma} \label{smoothness_problem2}
%The function $J^*(t,x)$ is differentiable with respect to $x$ for any $0 \leq t \leq 1$.
We have smooth fit: for all $0 \leq t < 1$,
\begin{align*}
 \lim_{x \uparrow B^* \sqrt{1-t}}\frac \partial {\partial x} J^*(t,x) &= \lim_{x \downarrow B^* \sqrt{1-t}} \frac \partial {\partial x} g(t,x), \\
\lim_{x \downarrow -B^* \sqrt{1-t}}\frac \partial {\partial x} J^*(t,x) &= \lim_{x \uparrow -B^* \sqrt{1-t}} \frac \partial {\partial x} g(t,x).
\end{align*}
\end{lemma}
\begin{proof} See Appendix \ref{appendix_proof}.
\end{proof}

\begin{lemma} \label{generator_prob2} (i) For $(t,x)$ such that $|x| < B^* \sqrt{1-t}$, we have $\mathcal{L} J^* (t,x)  = 0$.
(ii) For $(t,x)$ such that $|x| > B^* \sqrt{1-t}$,  we have $\mathcal{L} J^* (t,x) \leq 0$.
\end{lemma}
\begin{proof}
(i) This holds immediately in view of \eqref{def_candidate2} by the fact that $F_{2n+1}+G_{2n+1}$ solves the ODE \eqref{eq_ODE}.

(ii) For $x >  B^* \sqrt{1-t}$, as $J^*(t,x) = U(t,-x)+x^{2n+1}$ and because $\mathcal{L} U (t,-x) = 0$ for $-x <  0 < B^* \sqrt {1-t}$ in view of \eqref{expression_U_general},
\begin{align*}
\mathcal{L} J^* (t,x) = \mathcal{L} U (t,-x) + (2n+1)  \left[ n - \frac {x^2} {1-t} \right]  x^{2n-1}   =   - (2n+1)  \left[ \frac {x^2} {1-t}-n \right]  |x|^{2n-1}.
\end{align*}
For $x <  -B^* \sqrt{1-t}$, as $J^*(t,x) = U(t,x)-x^{2n+1}$ and because $\mathcal{L} U (t,x) = 0$ for $x <  0 < B^* \sqrt {1-t}$,
\begin{align*}
\mathcal{L} J^* (t,x) =  \mathcal{L} U (t,x) - (2n+1) \left[ n - \frac {x^2} {1-t} \right]   x^{2n-1}   =  - (2n+1) \left[ \frac {x^2} {1-t}-n \right]   |x|^{2n-1}.
\end{align*}
Hence, the proof is complete by Lemma \ref{lemma_root_n}.
\end{proof}

By Lemmas \ref{smoothness_problem2} and \ref{generator_prob2}, the verification of optimality is immediate.  We omit the proof of the following theorem because it is essentially the same as that of Theorem \ref{theorem_problem1}.
\begin{theorem}The function
$J^*$  is the value function.  Namely, $J(t, x) = J^*(t,x)$ for every $0 \leq t < 1$ and $x \in \R$; optimal stopping times are
\begin{align*}
\tau_1^* &:= \tau(B^*), \\ \tau_2^* &:= \left\{ \begin{array}{ll} \inf\{s\geq \tau(B^*): X_s\geq B^*\sqrt{1-s} \}, & \textrm{if } X_{\tau(B^*)} \leq -B^* \sqrt {1-\tau(B^*)}, \\
\inf\{s \geq \tau(B^*): X_s\leq -B^* \sqrt{1-s}\}, & \textrm{if } X_{\tau(B^*)} \geq B^* \sqrt {1-\tau(B^*)}.
\end{array} \right.   
\end{align*}
\end{theorem}

\section{Problem 3}  \label{section_problem3}
Our last problem is to solve, for fixed $q > 0$,
\begin{align}
W(t,x) := \sup_{t \leq \tau_1 \leq \tau_2< 1}\e_{t,x} \left[|X_{\tau_2}|^{q} -  |X_{\tau_1}|^{q} \right], \quad 0 \leq t< 1, \; x \in \R. \label{def_W}
\end{align}
%Let us see
%\[W(t,x) := \sup_{t \leq \tau_1 \leq \tau_2\leq 1}\e_{t,x} \left[|X_{\tau_2}|^{q} -  |X_{\tau_1}|^{q} \right], \quad 0 \leq t< 1, \; x \in \R, \]
%for some $q > 0$. This has applications when we think of a straddle or butterfly option or something.
%The one stage problem has been solved by \cite{Ekstrom}.  Recall the definition of stopping time $\tau(D)$ as in \eqref{def_tau_D}.
%By  \cite{Ekstrom}, for any $D > 0$,
%\begin{align*}
%\e_{t,x} [X_{\tau(D)}] = D^q \e_{t,x}[(  1-\tau(D))^{q/2} ] = (1-t)^{q/2} \frac {(G_{q}+F_{q})(x/\sqrt{1-t})} {(G_{q}+F_{q})(D)}, \quad -D \sqrt{1-t} < x < D \sqrt{1-t},
%\end{align*}
%which is maximized at $D^*$, which is a unique root of
%\begin{align*}
%0 = q-D^* (G_q'  + F_q') (D^*)/ (G_q + F_q) (D^*);
%\end{align*}
%see Figure \ref{plot_for_absolute_case} for the right hand side of the equation.  This value should be larger than $\sqrt{((q-1) \vee 0)/2}$.
By the strong Markov property, it can be written
\begin{align}
W(t,x) = \sup_{t \leq \tau < 1}\e_{t,x} \left[ h(\tau, X_\tau)\right], \quad 0 \leq t< 1, \; x \in \R, \label{W_single}
\end{align}
where
\begin{align*}
h(t,x):=\left\{\begin{array}{ll}
\overline{U}(t,x)-|x|^{q},& \mbox{if $|x|<D^*\sqrt{1-t}$},\\
0,&\mbox{if $|x|\geq D^*\sqrt{1-t}$}.
\end{array}\right.
\end{align*}
Here $\overline{U}(t,x)$ is the value function of the problem \eqref{ospbridge_general_absolute} and is written as \eqref{U_bar_expression}
with the same $D^*$ that satisfies \eqref{monotonicity_D}.

It is easily conjectured that the optimal stopping time for the problem  \eqref{W_single} is given by
\begin{align*}
\sigma(A) := \inf \{ s \geq t: |X_s| \leq A \sqrt{1-s}\}, 
\end{align*}
for some $A \in [0,D^*]$. Let us define its corresponding expected payoff by
\begin{align*}
W_A(t,x) &:= \e_{t,x} \left[ h(\sigma(A), X_{\sigma(A)})\right], \quad 0 \leq t< 1, \; x \in \R. 
\end{align*}
\begin{lemma} For all $(t,x)$ such that $|x| \geq A \sqrt{1-t}$,
\begin{align} \label{expression_W_A}
W_A(t,x) 
= (1-t)^{q/2} G_q(|x|/\sqrt{1-t})w(A),
\end{align}
where we define
\begin{align*}
w(A) :=\frac 1 {G_q(A)} \left[{(D^*)^q} \frac {(F_{q}+G_{q})(A)} {(F_{q}+G_{q})(D^*)} - A^q \right], \quad 0 \leq A \leq D^*.
\end{align*}
\end{lemma}
\begin{proof}
Suppose first that $x \geq A \sqrt{1-t}$ (then we must have $X_{\sigma(A)} = A \sqrt{1-\sigma(A)}$ a.s.). By \eqref{U_bar_expression},
\begin{align} \label{W_A_pre}
\begin{split}
W_A(t,x) 
%&= \e_{t,x} \left[(1-\sigma(A))^{q/2} (D^*)^q \frac {(F_{q}+G_{q})(A)} {(F_{q}+G_{q})(D^*)} - A^q(1-\sigma(A))^{q/2} \right] \\ 
&= \e_{t,x} \left[(1-\sigma(A))^{q/2} \right] \left[  {(D^*)^q}  \frac {(F_{q}+G_{q})(A)} {(F_{q}+G_{q})(D^*)} - A^q  \right],
\end{split}
\end{align}
and hence the problem boils down to computing the expectation on the right-hand side. As we have discussed in Section \ref{section_preliminaries}, we can write
\begin{align*}
\e_{t,x} \left[(1-\sigma(A))^{q/2} \right] = (1-t)^{q/2} \zeta(x / \sqrt{1-t}),
\end{align*}
where the function $\zeta$ satisfies the ODE \eqref{eq_ODE} with boundary conditions $\zeta(A)=1$ and $\lim_{y \rightarrow \infty}\zeta(y) = 0$.    A general solution of  \eqref{eq_ODE} is given by $\zeta(y) = \alpha F_q (y) + \beta G_q(y)$,
with the values of $\alpha$ and $\beta$ to be determined.
Because $\lim_{y \rightarrow \infty}F_q(y)=\infty$ and $\lim_{y \rightarrow \infty}G_q(y)=0$, we must have $\alpha = 0$. Solving
$\zeta(A)=1$, we have $\beta = {G_q(A)}^{-1}$. Plugging this in the right-hand side of  \eqref{W_A_pre} gives \eqref{expression_W_A}.

%Alternatively, because $\sigma(A) = \tau^-(A)$ $\p_{t,x}$ a.s.\ as $x > A \sqrt{1-t}$,
%\begin{align*}
%\e_{t,x} \left[(1-\sigma(A))^{q/2} \right] = \e_{t,x} \left[(1-\tau^-(A))^{q/2} \right] = \e_{t,-x} \left[(1-\tau^+(-A))^{q/2} \right], 
%\end{align*}
%which matches a generalization of \eqref{expectation_moments_2n_1}.

Finally, by symmetry, we have $W_A(t,x) = W_A(t,-x)$, $x \geq 0$.
Hence, the result can be extended to $x \leq -A \sqrt{1-t}$ as well.
\end{proof}

In view of \eqref{expression_W_A}, we shall maximize the function $w$. As shown in Figure \ref{plot_W}, $w$ admits a unique maximizer.  We shall show this analytically in Lemma  \ref{conjecture_A} below.  Notice that
\begin{align}
w(0)= \frac {2(D^*)^q} {(F_{q}+G_{q})(D^*)} > 0\quad \textrm{and} \quad w(D^*) = 0, \label{w_specific}
\end{align}
and, for all $A > 0$,
\begin{align} 
 w'(A) = \left[{(D^*)^q} \frac {(F_{q}+G_{q})'(A)} {(F_{q}+G_{q})(D^*)} - q A^{q-1} \right] \frac 1 {G_q(A)} - \left[{(D^*)^q} \frac {(F_{q}+G_{q})(A)} {(F_{q}+G_{q})(D^*)} - A^q \right] \frac {G_q'(A)} {(G_q(A))^2}. \label{w_prime}
\end{align}

% It can be confirmed that this is uniquely maximized.  In particular, when $q$ is small (at least when $q \leq 1$), then it is maximized at zero.   
 %Based on this heuristic observation, we write our conjecture below.
\begin{lemma} \label{conjecture_A}
(1) There exists a unique maximizer of $w$ over $[0, D^*]$, which we call $A^*$, such that
\begin{align}
A^* \leq \sqrt{(q-1)/2 \vee 0} \label{A_star_bound}
\end{align}
 and
\begin{align}
w'(A) \leq 0 \Longleftrightarrow A \geq A^*, \quad 0 \leq A \leq D^*. \label{w_prime_sign}
\end{align}
%(ii) In particular, $A^* = 0$ if $q \leq 1$.
(2) Moreover, $A^* = 0$ if and only if $q \leq 1$.
\end{lemma}
\begin{proof}
(1) For all $A > 0$, let us define
\begin{align*} 
L(A) &:= \frac {(F_{q}+G_{q})'(A) G_q(A) - {(F_{q}+G_{q})(A)} G_q'(A)} {(G_q(A))^2} = \frac {F_{q}'(A) G_q(A) - {F_{q}(A)} G_q'(A)} {(G_q(A))^2} > 0. 
\end{align*}
Then,  for any $A > 0$,
%\begin{align*} 
% w'(A) =  \frac {- q A^{q-1} } {G_q(A)} + \frac {A^qG_q'(A)} {(G_q(A))^2} + L(A) \frac {(D^*)^q} {(F_q+G_q) (D^*)}
%\end{align*}
%Hence,
\begin{align} \label{fraction_w_prime_L}
\begin{split}
\frac  {w'(A)} {L(A)} &=  \frac 1 {L(A)} \Big[ -\frac {q A^{q-1} } {G_q(A)} + \frac {A^qG_q'(A)} {(G_q(A))^2} \Big] + \frac {(D^*)^q} {(F_q+G_q) (D^*)} \\
&=  \frac {- q A^{q-1} G_q(A) + {A^qG_q'(A)}} {F_{q}'(A) G_q(A) - {F_{q}(A)} G_q'(A)}+ \frac {(D^*)^q} {(F_q+G_q) (D^*)}.
\end{split}
\end{align}
Because $F_q$ and $G_q$ satisfy the ODE \eqref{eq_ODE},
\begin{align*}
F_q''(A) &= A F_q'(A) + q F_q(A) \quad \textrm{and} \quad G_q''(A) = A G_q'(A) + q G_q(A). 
\end{align*} 
This gives
\begin{align*}
F_{q}''(A) G_q(A) - {F_{q}(A)} G_q''(A) &=  [A F_q'(A) + q F_q(A)]G_q(A) - F_q(A) [A G_q'(A) + q G_q(A)] \\
&= A [ F_{q}'(A) G_q(A) - {F_{q}(A)} G_q'(A) ].
\end{align*}
Now differentiating \eqref{fraction_w_prime_L} with respect to $A$, 
\begin{align*}
&[F_{q}'(A) G_q(A) - {F_{q}(A)} G_q'(A)]^2 \frac \partial {\partial A}\frac  {w'(A)} {L(A)}  \\
&=  [- q (q-1) A^{q-2} G_q(A) + {A^qG_q''(A)}] [F_{q}'(A) G_q(A) - {F_{q}(A)} G_q'(A)] \\
&\qquad -  [- q A^{q-1} G_q(A) + {A^qG_q'(A)}] [F_{q}''(A) G_q(A) - {F_{q}(A)} G_q''(A)] \\
&=  [- q (q-1) A^{q-2} G_q(A) + A^{q+1} G_q'(A) + q A^q G_q(A)] [F_{q}'(A) G_q(A) - {F_{q}(A)} G_q'(A)] \\
&\qquad -  [- q A^{q} G_q(A) + {A^{q+1}G_q'(A)}]  [F_{q}'(A) G_q(A) - {F_{q}(A)} G_q'(A)] \\
&=  2q A^{q-2}G_q(A) [F_{q}'(A) G_q(A) - {F_{q}(A)} G_q'(A)]   [A^2 - (q-1)/2]. 
\end{align*}
That is, $w'(A)/L(A)$ is differentiable and
\begin{align*}
\frac {F_{q}'(A) G_q(A) - {F_{q}(A)} G_q'(A)} {2q A^{q-2}G_q(A) } \frac \partial {\partial A}\frac  {w'(A)} {L(A)}  =   A^2 - \frac {q-1} 2.
\end{align*}
Notice that $F_{q}'(A) G_q(A) - {F_{q}(A)} G_q'(A)$ and $G_q(A)$ are both positive for $A > 0$.
Hence we see that
 \begin{align*}
 \frac \partial {\partial A} \frac  {w'(A)} {L(A)}  > 0 \Longleftrightarrow A >  \sqrt{(q-1)/2 \vee 0}, \quad A > 0.
 \end{align*}
Recall $D^*$ or the unique root of the equation \eqref{def_D}. The equivalence above together with $w'(D^*) = 0$ (and $w'(D^*)/ L(D^*) = 0$) and recalling that $D^* \geq  \sqrt{(q-1)/2 \vee 0}$ (as in Lemma  \ref{lemma_root_q}) and $L(\cdot) > 0$ shows that there exists a unique $A^* \in [0, \sqrt{(q-1)/2 \vee 0}]$ such that, for all $0 < A < D^*$, 
 \begin{align}
 w'(A) \leq  0 \Longleftrightarrow A \geq A^*, \label{w_L_sign}
 \end{align}
as desired.
 
(2) Suppose first that $q \leq 1$. Then  $A^* \leq  \sqrt{(q-1)/2 \vee 0} = 0$ and hence we must have $A^* = 0$.
Now suppose $q > 1$. Then, by taking a limit in \eqref{w_prime} and noticing that $(F_{q}+G_{q})'(0) = 0$,
\begin{align*} 
 w'(0+) = - {(D^*)^q} \frac {(F_{q}+G_{q})(0)} {(F_{q}+G_{q})(D^*)}  \frac {G_q'(0)} {(G_q(0))^2} > 0.
\end{align*}
This together with \eqref{w_L_sign} shows $A^* > 0$.

\end{proof}

%\begin{lemma}
%When $q=1$, $w'(A) \leq 0$ for all $A > 0$.
%\end{lemma}
%\begin{proof}
%We have
%\begin{align*}
%G_q(A) w'(A) &= \frac {(F_q+G_q)'(A)} {(F_q+G_q)(1)} - 1 - \Big[ \frac {(F_q+G_q)(A)} {(F_q+G_q)(1)}  - A \Big] \frac {G'(A)} {G(A)} \\
% &= \frac {\sqrt{2 \pi} A e^{A^2/2}} {\sqrt{2 \pi} e^{1/2}} - 1 - \Big[ \frac {\sqrt{2 \pi} e^{A^2/2}} {\sqrt{2 \pi} e^{1/2}} - A \Big] \Big[ \frac {\sqrt{2 \pi} A e^{A^2/2} \Phi(-A)-1} {\sqrt{2 \pi} e^{A^2/2} \Phi(-A)} \Big] \\
% &= A e^{(A^2-1)/2} - 1 - [e^{(A^2-1)/2} - A] [A - \frac 1 {\sqrt{2 \pi} e^{A^2/2} \Phi(-A)}] \\
% &= A^2 - 1 + \frac {e^{-1/2}} {\sqrt{2\pi} \Phi(-A)} - A \frac {e^{-A^2/2}} {\sqrt{2 \pi} \Phi(-A)}.
%\end{align*}
%Hence
%\begin{align*}
%\sqrt{2\pi} \Phi(-A)G_q(A) w'(A) = \sqrt{2\pi} \Phi(-A) (A^2 - 1) + e^{-1/2} - A e^{-A^2/2}.
%\end{align*}
%Its derivative with respect to $A$ becomes
%\begin{align*}
%\frac \partial {\partial A} [\sqrt{2\pi} \Phi(-A)G_q(A) w'(A) ] = \sqrt{2 \pi} \Phi(-A) 2 A > 0.
%\end{align*}
%This together with $w'(D^*) = w'(1) = 0$ and $\sqrt{2\pi} \Phi(-A)G_q(A) > 0$ shows the result.
%\end{proof}

\begin{figure}[htbp]
\begin{minipage}{1.0\textwidth}
\centering
\begin{tabular}{c}
 \includegraphics[scale=0.58]{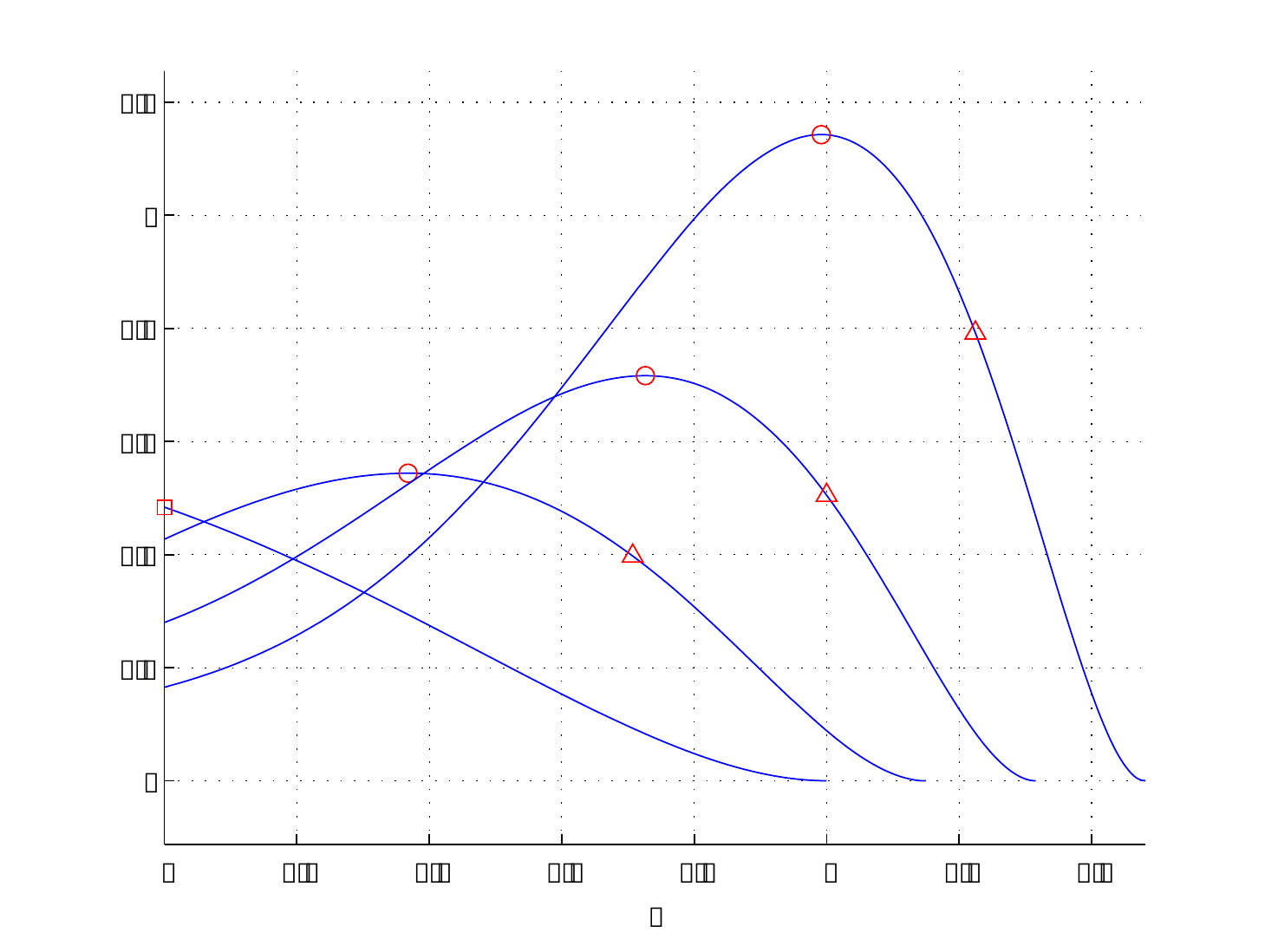}
\end{tabular}
\end{minipage}\caption{Plots of $w(A)$ on $[0, D^*]$ for $q=1,2,3,4$. For $q=2,3,4$, the circles (resp.\ triangles) indicate the points at $A^*$ (resp.\ $\sqrt{(q-1)/2}$).  For $q=1$, the square indicates the point at $A^* = \sqrt{(q-1)/2} = 0$.} \label{plot_W}
\end{figure}

Now we define
\begin{align}
W^*(t,x) := W_{A^*}(t,x)
&= \left\{ \begin{array}{ll}(1-t)^{q/2} G_q(|x|/\sqrt{1-t})w(A^*), & |x| > A^* \sqrt{1-t}, \\ h(t,x), & |x| \leq A^* \sqrt{1-t},  \end{array}\right. \label{def_candidate3}
\end{align}
as our candidate value function, and verify the optimality. Figure \ref{plot_prob3} plots $W^*$ and $h$ for $t=0$ as a function of $x$; this suggests Lemmas \ref{lemma_domination_prob3} and  \ref{smooth_fit_problem3}, which we shall prove analytically below. 

\begin{figure}[htbp]
\begin{minipage}{1.0\textwidth}
\centering
\begin{tabular}{cc}
 \includegraphics[scale=0.56]{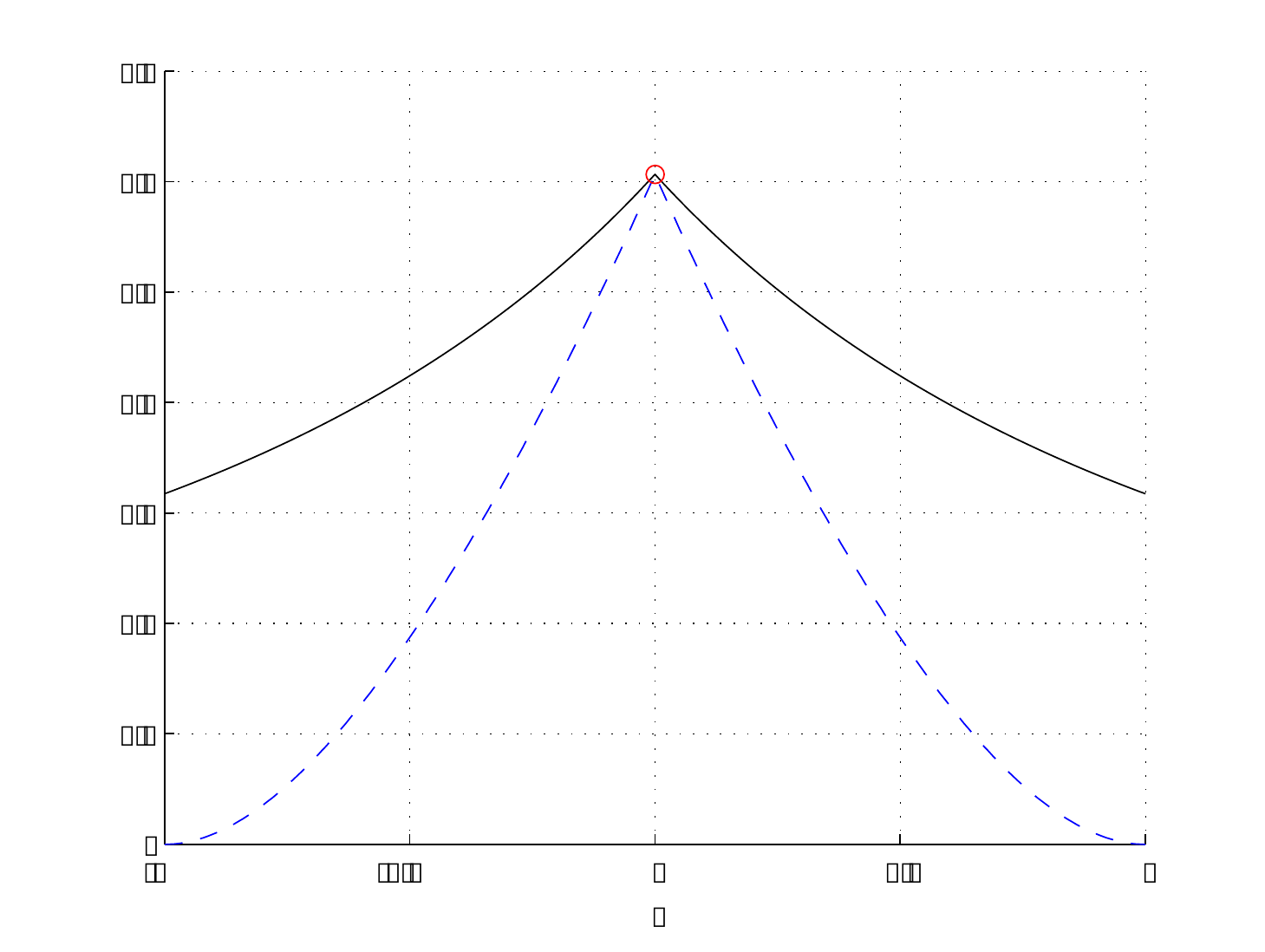} & \includegraphics[scale=0.56]{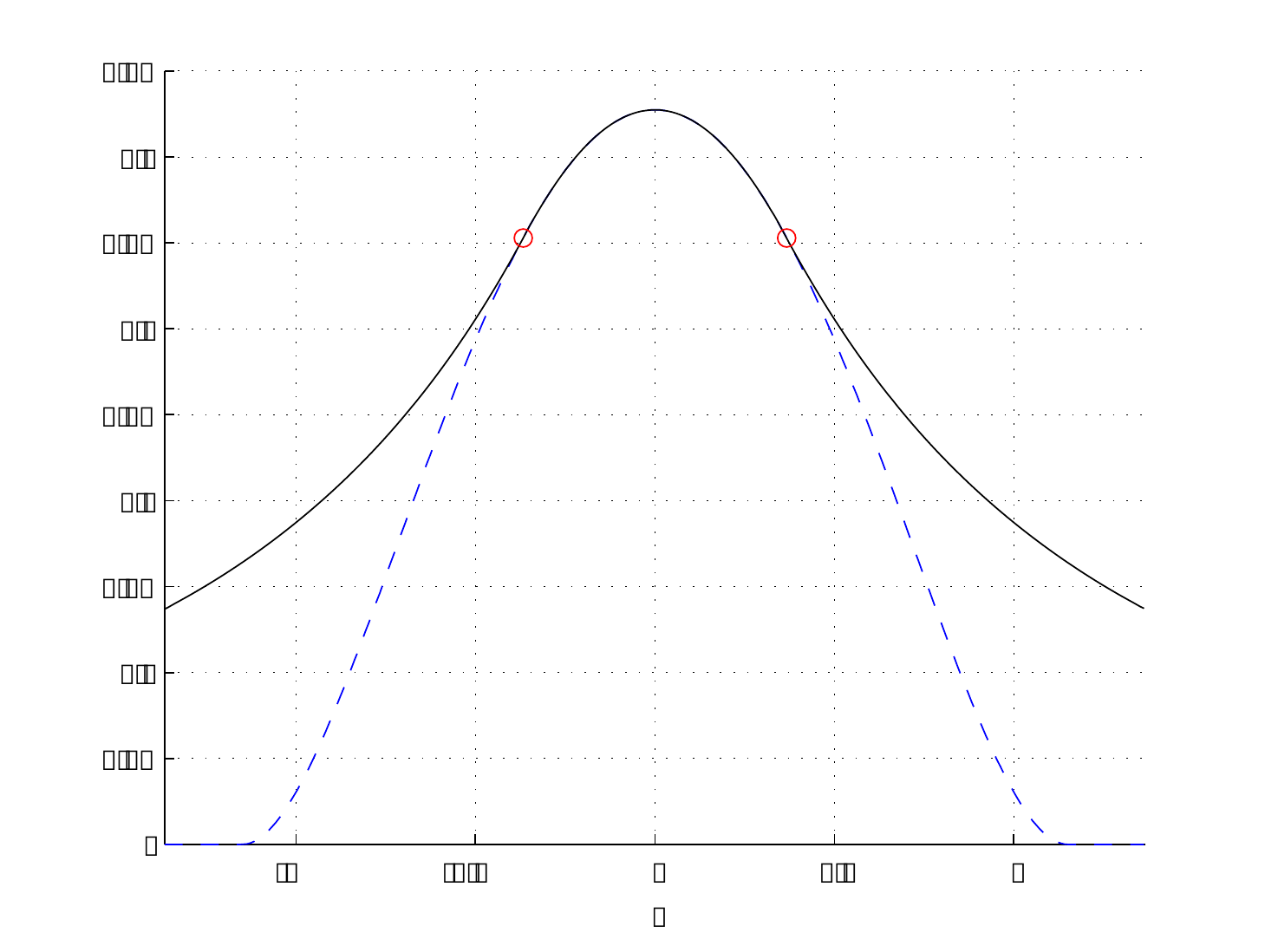} \\
$q=1$ & $q=2$ \\
 \includegraphics[scale=0.56]{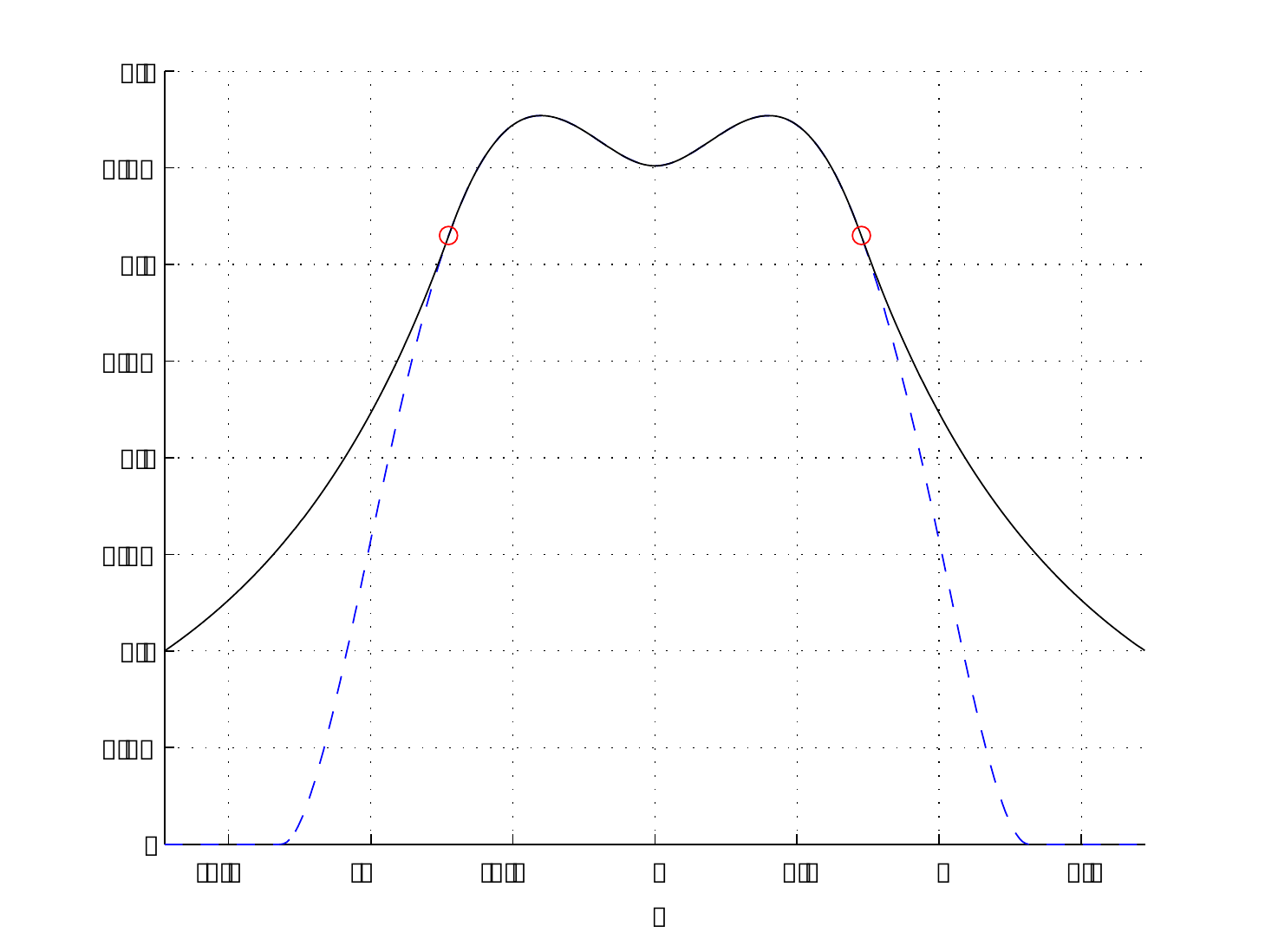} & \includegraphics[scale=0.56]{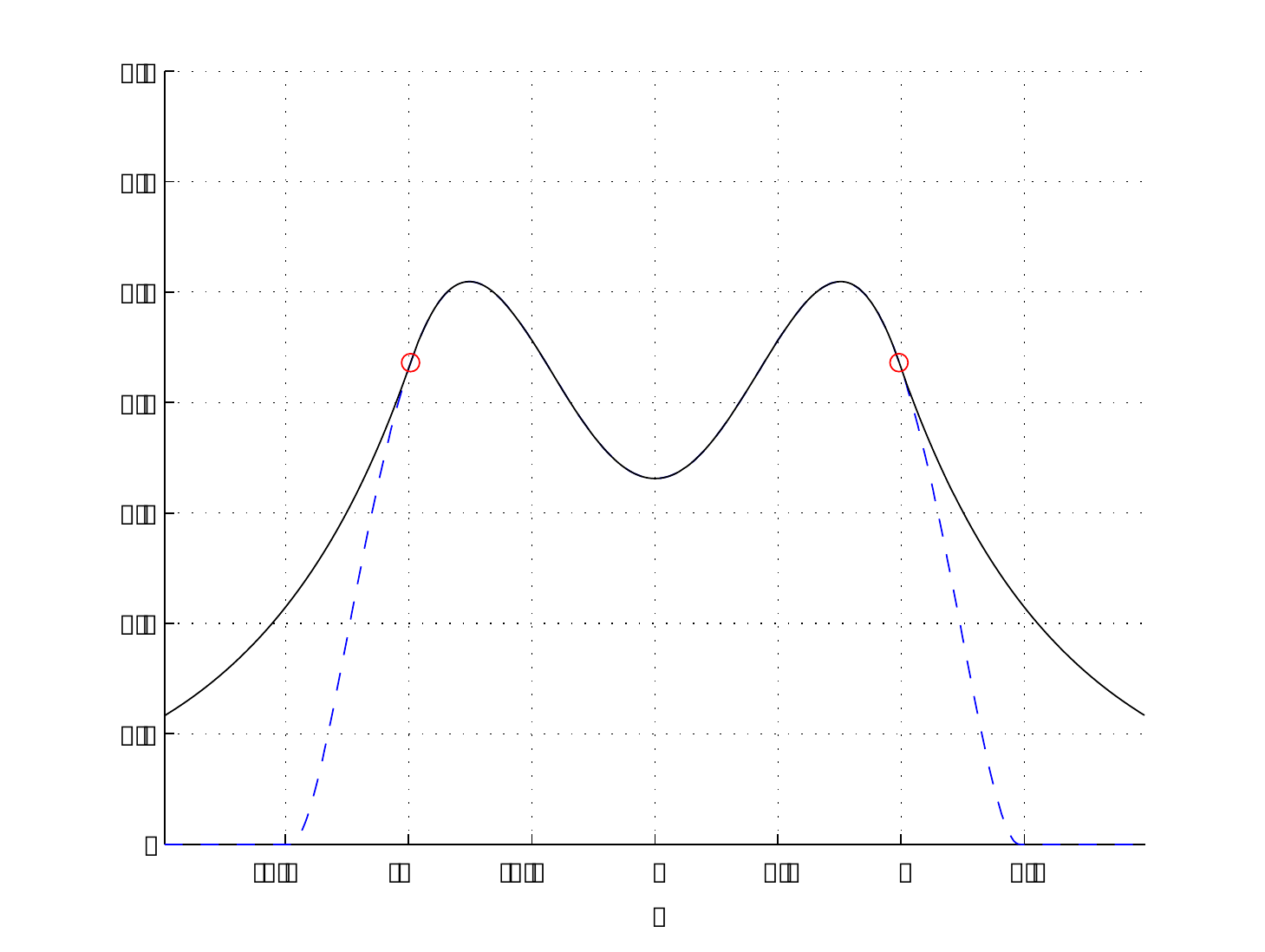} \\
$q=3$ & $q=4$ 
\end{tabular}
\end{minipage}\caption{Plots of $W^*(0,\cdot)$ (solid) and $h(0,\cdot)$ (dotted) for $q=1,2,3,4$. Circles indicate the points at $A^*$ and $-A^*$.} \label{plot_prob3}
\end{figure}

%Let $W^* (t,x) = W_{A^*} (t,x)$ for $0 \leq t \leq 1$ and $x \in \R$, as our candidate value function.

%We first have the following result; we shall omit the proof because it is similar to the ones for Lemma \ref{lemma_domination_prob1}.
\begin{lemma} \label{lemma_domination_prob3} We have $W^*(t,x) \geq h(t,x)$ for $0 \leq t < 1$ and $x \in \R$.
\end{lemma}
\begin{proof} Due to the symmetry of both $W^*$ and $h$ with respect to $x$, it is sufficient to show for $x \geq 0$.

When $x \geq D^* \sqrt{1-t}$, then $W^*(t,x) \geq 0 = h(t,x)$.  When $0 \leq x \leq A^*  \sqrt{1-t}$, then $W^*(t,x) = h(t,x)$ by definition.  
Finally, if $A^* \sqrt{1-t} < x < D^* \sqrt{1-t}$, due to continuous fit and the maximality of $A^*$ on $[0,\infty)$ we derive that
$W^*(t,x)=W_{A^*}(t,x)\geq W_{x/\sqrt{1-t}}(t,x) = h(t,x)$, as desired.
\end{proof}

Similarly to Problems 1 and 2, $W^*(t,x)$ is smooth enough to apply It\^o's formula at any $(t,x)$ such that $|x| \neq A^* \sqrt{1-t}$; see also \eqref{smoothness_at_0_if_A_great} below regarding the smoothness at $0$ when $A^* \neq 0$.  Regarding the smoothness on $|x| = A^* \sqrt{1-t}$, we have the following.
\begin{lemma} \label{smooth_fit_problem3} Fix $0 \leq t < 1$. (i)  If $A^* > 0$, then smooth fit 
\begin{align}\label{smooth_fit_condition_prob3}
\begin{split}
\lim_{x \downarrow A^* \sqrt{1-t}}\frac \partial {\partial x} W^*(t,x) &= \lim_{x \uparrow A^* \sqrt{1-t}} \frac \partial {\partial x} h(t,x), \\ \lim_{x \uparrow -A^* \sqrt{1-t}}\frac \partial {\partial x} W^*(t,x) &= \lim_{x \downarrow -A^* \sqrt{1-t}} \frac \partial {\partial x} h(t,x),  
\end{split}
\end{align}
holds.  (ii) If $A^*=0$, we have an inequality \begin{align}
\lim_{x \downarrow 0}\frac \partial {\partial x}W^*(t,x)  < \lim_{x \uparrow 0} \frac \partial {\partial x}W^*(t,x). \label{W_zero_inequality}
\end{align}
\end{lemma}
\begin{proof}
%(i) TODO We have
%\begin{align*}
%(G_q(A))^2 w'(A) = \left[{(D^*)^q} \frac {(F_{q}+G_{q})'(A)} {(F_{q}+G_{q})(D^*)} - q A^{q-1} \right] G_q(A) - \left[{(D^*)^q} \frac {(F_{q}+G_{q})(A)} {(F_{q}+G_{q})(D^*)} - A^q \right] G_q'(A).
%\end{align*}
%When $A^* > 0$, then the condition $W^{*'}(A^*)=0$ is equivalent to the smooth fit condition.

(i) We have
\begin{align*}
\lim_{x \downarrow A^* \sqrt{1-t}}\frac \partial {\partial x} W^*(t,x) = (1-t)^{(q-1)/2} \frac {G_q'(A^*)} {G_q(A^*)} \left[{(D^*)^q} \frac {(F_{q}+G_{q})(A^*)} {(F_{q}+G_{q})(D^*)} - (A^*)^q \right].
\end{align*}
On the other hand,
\begin{align*}
\lim_{x \uparrow A^* \sqrt{1-t}} \frac \partial  {\partial x} h(t,x)  = (1-t)^{(q-1)/2} \left[ (D^*)^q \frac {(F_{q}+G_{q})'(A^*)} {(F_{q}+G_{q})(D^*)}  - q (A^*)^{q-1} \right].
\end{align*}
When $A^* > 0$, Lemma \ref{conjecture_A} implies $w'(A^*)=0$. In view of \eqref{w_prime}, the two equations above are the same.
%\begin{align*}
%{(D^*)^q} \frac {(F_{q}+G_{q})'(A^*)} {(F_{q}+G_{q})(D^*)} - q (A^*)^{q-1}   = \left[{(D^*)^q} \frac {(F_{q}+G_{q})(A^*)} {(F_{q}+G_{q})(D^*)} - (A^*)^q \right] \frac {G_q'(A^*)} {G_q(A^*)}.
%\end{align*}
Hence the first equality of \eqref{smooth_fit_condition_prob3} holds.  The proof of the second equality holds by symmetry.

(ii) Now suppose $A^* = 0$.  In this case, by \eqref{w_specific},
\begin{align*}
W^*(t,x)
= 2 (1-t)^{q/2}  {(D^*)^q} \frac {G_q(|x|/\sqrt{1-t})} {(F_{q}+G_{q})(D^*)},
\end{align*}
and hence taking derivatives and then limits,
\begin{align*}
\lim_{x \downarrow 0}\frac \partial {\partial x}W^*(t,x)
&= 2(1-t)^{(q-1)/2} {(D^*)^q}  \frac {G_q'(0)}  {(F_{q}+G_{q})(D^*)}, \\
\lim_{x \uparrow 0} \frac \partial {\partial x}W^*(t,x)
&= - 2 (1-t)^{(q-1)/2} {(D^*)^q}  \frac {G_q'(0)}  {(F_{q}+G_{q})(D^*)}.
\end{align*}
Because $F_q$ and $G_q$ are nonnegative and $G_q'$ is negative, we have \eqref{W_zero_inequality}, as desired.
%\begin{align*}
%\frac \partial {\partial x}W^*(t,0+)  \leq \frac \partial {\partial x}W^*(t,0-).
%\end{align*}
\end{proof}

%\begin{lemma} \label{lemma_root_A}If Conjecture \ref{conjecture_A} holds, then we have $A^* \leq \sqrt{(q-1)/2 \vee 0}$.
%\end{lemma}
%\begin{proof} See Appendix \ref{appendix_proof}.
%\end{proof}

\begin{lemma} \label{generator_prob3} (i) For $(t,x)$ such that $|x| > A^* \sqrt{1-t}$, we have $\mathcal{L} W^* (t,x)  = 0$.
(ii) If $A^* > 0$, for $(t,x)$ such that $0 < |x| < A^* \sqrt{1-t}$,  we have $\mathcal{L} W^* (t,x) \leq 0$.
\end{lemma}
\begin{proof}
(i) It is clear by the fact that $G_q$ solves the ODE \eqref{eq_ODE} in view of \eqref{def_candidate3}.

(ii) By Lemma \ref{conjecture_A}(2), $A^* > 0$ guarantees $q > 1$. Because $|x| < A^* \sqrt{1-t} < D^* \sqrt{1-t}$, we must have  $\mathcal{L} \overline{U} (t,x) = 0$ in view of \eqref{U_bar_expression}.

For $0 < x <  A^* \sqrt{1-t}$, as $W^*(t,x) = \overline{U}(t,x)-x^{q}$, we have  
\begin{align*}
\mathcal{L} W^* (t,x) = - \left[ (q-1)/ 2 -  {x^2} /(1-t) \right]  qx^{q-2}
\end{align*}
 whereas, for $-A^* \sqrt{1-t} < x <  0$,  we have $\mathcal{L} W^* (t,x) =   -\left[ (q-1)/ 2 -  {x^2} /(1-t) \right]  q|x|^{q-2}$.
Now  \eqref{A_star_bound} completes the proof.
\end{proof}

We now have the following optimality results by Lemmas \ref{lemma_domination_prob3}, \ref{smooth_fit_problem3} and \ref{generator_prob3}.  An important difference with the verification of Problems 1 and 2 is the potential non-differentiability at $x = 0$, but this does not cause any issue.

Recall from \eqref{W_zero_inequality} that, for the case $A^* = 0$,  the smooth fit at $A^*$ fails. However, this can be resolved easily by using the following version of It\^o's formula (see, e.g., \cite{Peskir}), for all $t \leq u < 1$,
\begin{multline*}
W^*(u, X_u)  = W^*(t, x)  + \int_t^u \mathcal{L} W^*(s, X_s) 1_{\{ X_s \neq  0 \}} \diff s \\ + \frac 1 2 \int_t^u  \Big(\lim_{z \downarrow 0}\frac \partial  {\partial z}W^*(s, z)- \lim_{z \uparrow 0} \frac \partial  {\partial z} W^*(s, z) \Big) \diff l_s^{0} + \int_t^u  \frac \partial {\partial x} W^*(s, X_s)  1_{\{ X_s \neq 0 \}} \diff W_s,
\end{multline*}
where $\{ l_s^0\}_{t \leq s < 1}$ denotes the local time of $X$ at $0$.  The supermartingale property of the process $\{ W^*(u, X_u) \}_{t \leq u \leq 1}$ still holds by \eqref{W_zero_inequality} and Lemma \ref{generator_prob3}.  For the case $A^* > 0$, Lemma \ref{conjecture_A}(2) guarantees that $q > 1$; in this case,  
\begin{align}
\lim_{z \downarrow 0}\frac \partial  {\partial z}W^*(s, z)=\lim_{z \downarrow 0}\frac \partial  {\partial z}h(s, z)= \lim_{z \uparrow 0} \frac \partial  {\partial z} h(s, z)= \lim_{z \uparrow 0} \frac \partial  {\partial z} W^*(s, z), \label{smoothness_at_0_if_A_great}
\end{align}
by using the equality $(F_{q}+G_{q})'(0) =0$ and $x^{q-1} \rightarrow 0$ as $x \rightarrow 0$.
The rest of the proof is omitted as it is similar to that of Theorem \ref{theorem_problem1}.

\begin{theorem} The function
$W^*$  is the value function.  Namely, $W(t, x) = W^*(t,x)$ for every $0 \leq t < 1$ and $x \in \R$; optimal stopping times are
\begin{align*}
\tau_1^* := \sigma(A^*) \quad \textrm{and} \quad \tau_2^* := \inf\{s \geq \sigma(A^*) : |X_s|\geq  D^* \sqrt{1-s}\}.
\end{align*}
\end{theorem}

\appendix

\section{Proofs} \label{appendix_proof}

\begin{proof}[Proof of Lemma \ref{lemma_root_n}]
The case $n = 0$ is trivial (as $B^* > 0$) and hence we shall focus on the case $n \geq 1$.

Assume for contradiction that $B^* <  \sqrt{n}$.  Then we can take $(t,x)$ such that $x/\sqrt{1-t} \in (B^*, \sqrt{n})$.
By It\^o's formula
\begin{align}
\diff X^{2n+1}_s = (2n+1) \left[  n-\frac {X_s^2} {1-s}   \right]  X_s^{2n-1} \diff s + (2n+1) X_s^{2n} \diff W_s, \quad  t \leq s < 1. \label{eq_ito_2n_1}
\end{align}
Define the first downcrossing time of $X$:
\begin{align}
T^-(\delta):=\inf\{u \geq t: X_u \leq \delta \}, \quad \delta \geq 0. \label{def_T_delta}
\end{align}
Fix $0 < \varepsilon < x$.
For  $s \in [t, \tau^+(\sqrt{n}) \wedge T^-(\varepsilon)]$,  we have $\varepsilon / \sqrt{1-s} \leq X_s / \sqrt{1-s} \leq \sqrt{n}$ (and hence
$\varepsilon  \leq X_s \leq \sqrt{n} \sqrt{1-s}$). Therefore, taking expectation of the integral of \eqref{eq_ito_2n_1} gives
\begin{align*}
\e_{t,x} [X^{2n+1}_{\tau^+(\sqrt{n}) \wedge T^-(\varepsilon)}] &= x^{2n+1} + \e_{t,x}  \left[ \int_t^{\tau^+(\sqrt{n}) \wedge T^-(\varepsilon)} (2n+1) \left[  n-\frac {X^2_s} {1-s}   \right]  X_s^{2n-1} \diff s \right] \\ &\geq x^{2n+1}.
\end{align*}
Dominated convergence gives upon $\varepsilon \downarrow 0$ that $\e_{t,x} [X^{2n+1}_{\tau^+(\sqrt{n}) \wedge T^-(0)}] \geq x^{2n+1}$.
Moreover, because $X^{2n+1}_{\tau^+(\sqrt{n})} \geq X^{2n+1}_{\tau^+(\sqrt{n}) \wedge T^-(0)}$ a.s., we also have
$\e_{t,x} [X^{2n+1}_{\tau^+(\sqrt{n})}]  \geq x^{2n+1}$.

On the other hand,  by  \eqref{expectation_moments_2n_1} and \eqref{monotonicity_B}, and because $\sqrt{n} > x/\sqrt{1-t} > B^*$, we must have
\begin{align*}
x^{2n+1} = \e_{t,x} [X_{\tau^+(x/\sqrt{1-t})}^{2n+1}]  > \e_{t,x} [X_{\tau^+(\sqrt{n})}^{2n+1}].
\end{align*}
This is a contradiction, and hence $B^* \geq \sqrt{n}$.
\end{proof}

\begin{proof}[Proof of Lemma \ref{lemma_root_q}]  Because the case $q \leq 1$ is trivial, we focus on the case $q > 1$. 
Assume for contradiction that $D^* < \sqrt{(q-1)/2}$.  Then we can take $(t,x)$ such that $x/\sqrt{1-t} \in (D^*, \sqrt{(q-1)/2})$.

Arguments similar to the ones in the proof of Lemma \ref{lemma_root_n} give
\begin{align*}
\e_{t,x} [|X_{\tau( \sqrt{(q-1)/2}) \wedge T^-(0)}|^{q}] = \e_{t,x} [X^{q}_{\tau( \sqrt{(q-1)/2}) \wedge T^-(0)}] \geq x^q,
\end{align*}
where $T^-(0)$ is defined as in \eqref{def_T_delta}.
%
%By Ito's formula
%\begin{align*}
%\diff X^{q}_s = \left[  \frac {q-1} 2-\frac {X_s^2} {1-s}   \right] q X_s^{q-2} \diff s + q X_s^{q-1} \diff W_s, \quad  t \leq s < 1.
%\end{align*}
%
%For  $s \in (t, \tau^+(\sqrt{(q-1)/2}) \wedge \tau^-(\varepsilon) \wedge (1-\delta))$,  we have $\varepsilon \leq X_s / \sqrt{1-s} \leq \sqrt{(q-1)/2}$ and
%$\varepsilon \sqrt{\delta}  \leq \varepsilon \sqrt{1-s} \leq X_s \leq \sqrt{n} \sqrt{1-s}$.
%
%Because, on $(t, \tau^+(  \sqrt{(q-1)/2}) \wedge \tau^-(0))$,  $0 \leq X_s  \leq  \sqrt{(q-1)/2} \sqrt{1-s} \leq  \sqrt{(q-1)/2}$, we have
%\begin{align*}
%\e_{t,x} [|X_{\tau^+( \sqrt{(q-1)/2}) \wedge \tau^-(0)}|^{q}] &= \e_{t,x} [X^{q}_{\tau^+( \sqrt{(q-1)/2}) \wedge \tau^-(0)}] \\ &= x^{q} + \e_{t,x}  \left[ \int_t^{\tau^+( \sqrt{(q-1)/2}) \wedge \tau^-(0)} \left[  \frac {q-1} 2-\frac {X_s^2} {1-t}   \right] q X_s^{q-2}  \right] \geq x^{q}.
%\end{align*}
Moreover, because $|X_{\tau(\sqrt{(q-1)/2})}|^{q} \geq X^q_{\tau(\sqrt{(q-1)/2}) \wedge T^-(0)}$ a.s., we also have
$\e_{t,x} [|X_{\tau(\sqrt{(q-1)/2})}|^{q} ]  \geq x^{q}$.

On the other hand,  by  \eqref{moment_x_q} and \eqref{monotonicity_D}, and because $\sqrt{(q-1)/2} > x/\sqrt{1-t} > D^*$, we must have
\begin{align*}
x^{q} = \e_{t,x} [|X_{\tau(x/\sqrt{1-t})}|^{q}]  > \e_{t,x} [|X_{\tau(\sqrt{(q-1)/2})}|^{q}].
\end{align*}
This is a contradiction, as desired.
\end{proof}

\begin{proof}[Proof of Lemma \ref{lemma_smooth_fit}]
%It is clear that the the differentiability holds for any $x \neq C^* \sqrt{1-t}$.  Hence it remains to show the smooth fit property:
%\begin{align}
%0 = \lim_{x \downarrow C^* \sqrt{1-t}}\frac \partial {\partial x} V^*(t,x) - \lim_{x \uparrow C^* \sqrt{1-t}} \frac \partial {\partial x} f(t,x).  \label{smooth_fit_condition}
%\end{align}
Differentiating \eqref{def_candidate},
\begin{align*}
\frac \partial {\partial x}V^*(t,x)
=v(C^*) \sqrt{2 \pi} \Big(  \frac x {\sqrt{1-t}} e^{{x^2}/ {(2(1-t))}} \Phi (- x /{\sqrt{1-t}} ) - \frac 1 {\sqrt{2 \pi}}\Big),
\end{align*}
and hence
\begin{align*}
&\lim_{x \downarrow C^* \sqrt{1-t}}\frac \partial {\partial x}V^*(t,x) = v(C^*) \sqrt{2 \pi} \Big(   {C^*} e^{{(C^*)^2}/ {2}} \Phi (-C^* ) - \frac 1 {\sqrt{2 \pi}}\Big) \\
%\end{align*}
%and hence
%\begin{eqnarray*}
%&&\lim_{x \downarrow C^* \sqrt{1-t}}\frac \partial {\partial x}V^*(t,x)\\
%&=&\left( \sqrt{2\pi}(1-(B^*)^2)e^{(C^*)^2/2}\frac{\Phi(C^*)}{C^*} -1\right)\frac{C^*}{\Phi(-C^*)}e^{-(C^*)^2/2}\sqrt{1-t}  \Big( \frac {C^*} {\sqrt{1-t}} e^{ (C^*)^2 /2} \Phi (- C^*) - \frac 1 {\sqrt{(1-t) 2 \pi}}\Big) \\
%&=&\left( \sqrt{2\pi}(1-B^2)e^{C^2/2}\frac{\Phi(C)}{C} -1\right)\frac{C}{\Phi(-C)}e^{-C^2/2}  \Big(  C e^{ C^2 /2} \Phi (- C) - \frac 1 {\sqrt{2 \pi}}\Big) \\
%&=\left( \sqrt{2\pi}(1-(B^*)^2)e^{(C^*)^2/2}{\Phi(C^*)} -{C^*}\right)  \Big(  C^*- \frac 1 {\sqrt{2 \pi}} \frac 1 {\Phi(-C^*)}e^{-(C^*)^2/2} \Big) \\
&= C^* \sqrt{2\pi}(1-(B^*)^2)e^{(C^*)^2/2}{\Phi(C^*)} -(C^*)^2 -  \frac {\Phi(C^*)} {\Phi(-C^*)} (1-(B^*)^2) +C^* \frac 1 {\sqrt{2 \pi}} \frac 1 {\Phi(-C^*)}e^{-(C^*)^2/2}.
\end{align*}
On the other hand,
%\begin{align*}
%\frac \partial {\partial x} f(t,x) & = -1 + \frac \partial {\partial x} \Big[\sqrt{2 \pi (1-t)} (1-(B^*)^2) e^{{x^2} /{(2 (1-t))}} \Phi \Big(\frac x {\sqrt{1-t}} \Big) \Big] \\
%&= -1 + \sqrt{2 \pi (1-t)} (1-(B^*)^2)\Big( \frac x {1-t} e^{{x^2} / {(2 (1-t))}} \Phi \Big(\frac x {\sqrt{1-t}} \Big) + \frac 1 {\sqrt{(1-t)2 \pi}}\Big),
%\end{align*}
%and hence
\begin{align*}
\lim_{x \uparrow C^* \sqrt{1-t}} \frac \partial {\partial x} f(t,x)
%&= -1 + \sqrt{2 \pi (1-t)} (1-(B^*)^2)\Big( \frac {C^*} {\sqrt{1-t}} e^{{(C^*)^2} / {2}} \Phi (C^*) + \frac 1 {\sqrt{(1-t)2 \pi}}\Big) \\
&=  \sqrt{2 \pi} (1-(B^*)^2) C^* e^{{(C^*)^2} /2} \Phi (C^*)  - (B^*)^2.
\end{align*}
Taking the difference between the two
\begin{align*}
&\lim_{x \downarrow C^* \sqrt{1-t}}\frac \partial {\partial x}V^*(t,x)  - \lim_{x \uparrow C^* \sqrt{1-t}} \frac \partial {\partial x} f(t,x) \\
%&-\sqrt{2 \pi} (1-(B^*)^2) C^* e^{ {(C^*)^2} /2} \Phi (C^*)  + (B^*)^2 \\  &+C^* \sqrt{2\pi}(1-(B^*)^2)e^{(C^*)^2/2}{\Phi(C^*)} -(C^*)^2 -  \frac {\Phi(C^*)} {\Phi(-C^*)} (1-(B^*)^2) +C^* \frac 1 {\sqrt{2 \pi}} \frac 1 {\Phi(-C^*)}e^{-(C^*)^2/2} \\
%&= (B^*)^2 -(C^*)^2 -  \frac {\Phi(C^*)} {\Phi(-C^*)} (1-(B^*)^2) +C^* \frac 1 {\sqrt{2 \pi}} \frac 1 {\Phi(-C^*)}e^{-(C^*)^2/2} \\
&= 1  -(C^*)^2 -  \Big(1+\frac {\Phi(C^*)} {\Phi(-C^*)} \Big) (1-(B^*)^2) +C^* \frac 1 {\sqrt{2 \pi}} \frac 1 {\Phi(-C^*)}e^{-(C^*)^2/2} \\
&= \frac 1 {\Phi(-C^*)} \Big[ {\Phi(-C^*)} (1  -(C^*)^2) -  ({\Phi(-C^*)}+ {\Phi(C^*)} ) (1-(B^*)^2) +C^* \frac 1 {\sqrt{2 \pi}} e^{-(C^*)^2/2} \Big]\\
%&= \frac 1 {\Phi(-C^*)}  \Big[ {\Phi(-C^*)} (1  -(C^*)^2) -   (1-(B^*)^2) +C^* \frac 1 {\sqrt{2 \pi}} e^{-(C^*)^2/2} \Big] 
&= - \frac {u(C^*)} {\Phi(-C^*)},
\end{align*}
which equals $0$ thanks to our choice of $C^*$ as  in Lemma \ref{lemma_maximality_c_star}.
\end{proof}

\begin{proof}[Proof of Lemma \ref{smoothness_problem2}]
We show the differentiability at $x = B^* \sqrt{1-t}$ (that of $x = -B^* \sqrt{1-t}$ holds by symmetry). Differentiating \eqref{def_candidate2} with respect to $x$ gives
\begin{align*}
\frac \partial {\partial x}J^*(t,x)
&= (1-t)^{n} {(F_{2n+1}+G_{2n+1})'(x/\sqrt{1-t})}j(B^*),
\end{align*}
and by Remark \ref{remark_j},
\begin{align}
\lim_{x \uparrow B^* \sqrt{1-t}}\frac \partial {\partial x}J^*(t,x)
%&= (1-t)^{n} \frac {{(F_{2n+1}+G_{2n+1})'(B^*)}} {(F_{2n+1}+G_{2n+1})(B^*)}\left[ (B^*)^{2n+1} + (B^*)^{2n+1}\frac {\red{G_{2n+1} (B^*)}} {F_{2n+1}(B^*)} \right] \nonumber \\
%&= (1-t)^{n} (B^*)^{2n+1}\frac {{(F_{2n+1}+G_{2n+1})'(B^*)}} {(F_{2n+1}+G_{2n+1})(B^*)} \frac {\red{(F_{2n+1}+G_{2n+1})(B^*)}} {F_{2n+1}(B^*)}  \nonumber \\
&= (1-t)^{n} (B^*)^{2n+1} \frac {{(F_{2n+1}+G_{2n+1})'(B^*)}} {F_{2n+1}(B^*)}  \nonumber \\
&= (1-t)^{n} (B^*)^{2n+1} \frac {F_{2n+1}'(B^*)-F_{2n+1}'(-B^*)} {F_{2n+1}(B^*)} . \label{J_diff_limit}
\end{align}
On the other hand,
\begin{align*}
\frac \partial {\partial x}g (t,x)
&=-(1-t)^{n}(B^*)^{2n+1}\frac {F_{2n+1}' (-x/\sqrt{1-t})} {F_{2n+1}(B^*)}+  (2n+1) x^{2n}.
\end{align*}
%where
%\begin{align*}
%U'(t, x) = (1-t)^{n}(B^*)^{2n+1}\frac {F_{2n+1}' (x/\sqrt{1-t})} {F_{2n+1}(B^*)}.
%\end{align*}
Hence,
\begin{align*}
\lim_{x \downarrow B^* \sqrt{1-t}} \frac \partial {\partial x}g (t,x)
&=- (1-t)^{n}(B^*)^{2n+1} \left[ \frac {F_{2n+1}' (-B^*)} {F_{2n+1}(B^*)}-  \frac {2n+1} {B^*} \right],
\end{align*}
which equals \eqref{J_diff_limit} by \eqref{B_star}, as desired.
%\begin{align*}
%\frac {B^* F'_{2n+1}(B^*)} {F_{2n+1}(B^*)} = 2n+1.
%\end{align*}
%\begin{align*}
%\lim_{x \downarrow B^* \sqrt{1-t}} \frac \partial {\partial x}g (t,x)
%&=- (1-t)^{n}(B^*)^{2n+1}\frac {F_{2n+1}' (-B^*)} {F_{2n+1}(B^*)}+  \frac {B^* F'_{2n+1}(B^*)} {F_{2n+1}(B^*)}  ( B^* \sqrt{1-t})^{2n}.
%\begin{align*}
%\lim_{x \downarrow B^* \sqrt{1-t}} \frac \partial {\partial x}g (t,x)
%&=- (1-t)^{n}(B^*)^{2n+1}\frac {F_{2n+1}' (-B^*)} {F_{2n+1}(B^*)}+  \frac {B^* F'_{2n+1}(B^*)} {F_{2n+1}(B^*)}  ( B^* \sqrt{1-t})^{2n}.
%\end{align*}
%Hence, these indeed match.
\end{proof}

\section*{Acknowledgements} 
The authors thank the anonymous referee for his/her insightful comments.
E.\ J.\ Baurdoux was visiting CIMAT, Guanajuato when part of this work was carried out and he is grateful for their hospitality and support.
B.\ A.\ Surya acknowledges the support by the Department of IEOR of Columbia University  during his stay. K. Yamazaki is in part supported by MEXT KAKENHI Grant Number  26800092, the Inamori foundation research grant, and the Kansai University Subsidy for Supporting young Scholars 2014.

\end{document}